\documentclass[12pt, reqno, a4paper]{amsart}
\usepackage{times}

\usepackage[utf8]{inputenc}
\usepackage[USenglish]{babel}
\usepackage{amsmath,amsthm,amssymb,amsfonts}
\usepackage{mathtools}
\usepackage{mathrsfs}
\usepackage{bbm}
\usepackage{booktabs}

\usepackage{indentfirst}
\usepackage{enumitem}
\usepackage{xcolor}

\usepackage{float}

\usepackage[breaklinks=true, bookmarksopenlevel=1, bookmarksdepth=2]{hyperref}
\hypersetup{
colorlinks,
   linkcolor={cyan!80!black},
   citecolor={cyan!80!black},
 urlcolor={cyan!80!black}
}

\allowdisplaybreaks

\newtheorem{thm}{Theorem}[section]

\newtheorem{lem}[thm]{Lemma}
\newtheorem{prop}[thm]{Proposition}

\theoremstyle{plain} % just in case the style had changed
\newcommand{\thistheoremname}{}
\newtheorem*{genericthm}{\thistheoremname}

\theoremstyle{definition}

\theoremstyle{remark}
\newtheorem{rem}[thm]{Remark}

\newtheorem*{ntt}{Notation}

\numberwithin{equation}{section}

\newcommand{\N}{\mathbb{N}}      % N = Naturals
\newcommand{\Z}{\mathbb{Z}}      % Z = Integers
\newcommand{\R}{\mathbb{R}}      % R = Reals
\newcommand{\eps}{\varepsilon}   % epsilon

\renewcommand{\pmod}[1]{         % (mod #1)
  ~(\mathrm{mod}~#1)}

\renewcommand{\P}{\mathbb{P}}

\usepackage[margin=1in]{geometry}

\restylefloat{table}

\setlist[itemize]{leftmargin=*}
\setlist[enumerate]{leftmargin=*}

\begin{document}

\title{Thin subbases of Piatetski-Shapiro sequences}%

\author{Christian T\'afula}%
\address{Instituto de Matem\'atica, Estat\'istica\\
e Ci\^encia da Computa\c{c}\~ao\\
Universidade de S\~ao\break Paulo\\
Rua do Mat\~ao, 1010\\
S\~ao Paulo, SP 05508-090\\
Brazil}
\curraddr{}
\email{tafula@ime.usp.br}
\thanks{}


% ----------------------------------------------------------------
 \begin{abstract}
  For a non-integral real number $c>1$, let $\mathbb{N}_{(c)}:=\{\lfloor n^c\rfloor ~|~ n\in\mathbb{N}\}$. We show that $\mathbb{N}_{(c)}$ contains thin subbases of every order $h\geq 5$ when $1<c<2$, and $h\geq (\lfloor 2c\rfloor+1)(\lfloor 2c\rfloor+2)+1$ when $c>2$. In fact, for every regularly varying function $F$ such that 
  \[ \frac{F(x)}{\log x}\to\infty\quad\text{ and } \quad F(x)\leq (1+o(1))\frac{\Gamma(1+1/c)^h}{\Gamma(h/c)} x^{h/c-1}, \]
  there exists $A\subseteq\mathbb{N}_{(c)}$ with $r_{A,h}(n)\sim F(n)$. We also establish analogous results for $k$-th powers of Piatetski-Shapiro numbers and Piatetski-Shapiro primes for small $c$. %The argument combines a probabilistic subbasis criterion with Hua-type mean value bounds and weighted circle method asymptotics.
 \end{abstract}

\maketitle
% ----------------------------------------------------------------

%%%%%%%%%%%%%%%%%%%%%%%%%%%%%%%%%%%%%%%%%%%%%%%%%%%%%%%%%
\section{Introduction}
 A central question in additive number theory asks whether a given set $A \subseteq \mathbb{N}$ can serve as an additive basis for the integers or for particular congruence classes. For an integer $h \geq 2$ and a subset $A\subseteq \mathbb{N}$, define the representation function
 \begin{equation}
  r_{A,h}(n) := \#\{(x_1, \ldots, x_h) \in A^h ~|~ x_1 + \cdots + x_h = n\}, \label{repfunc}
 \end{equation}  
 which counts the number of ordered $h$-tuples in $A$ summing to $n$. Classical problems in this area include Waring’s problem, which studies $\mathbb{N}^k := \{n^k ~|~ n \in \mathbb{N}\}$, and the Waring--Goldbach problem, concerning $\mathbb{P}^k := \{p^k ~|~ p \text{ prime}\}$. These problems ask whether such sets represent all sufficiently large integers, and if so, seek asymptotic formulas for $r_{A,h}(n)$.

 An intriguing direction asks how ``thin'' a set $A\subseteq \mathbb{N}$ can be while still representing all large integers. Inspired by a problem of Sidon, Erd\H{o}s \cite{erdo56} used probabilistic methods to show the existence of $A \subseteq \mathbb{N}$ with $r_{A,2}(n) \asymp \log n$. Erd\H{o}s and Tetali \cite{erdtet90} later extended this to all $h \ge 2$, showing that there exists $A \subseteq \mathbb{N}$ with $r_{A,h}(n) \asymp \log n$. By a strong form of the Erd\H{o}s--Tur\'an conjecture for additive bases, one expects that for any $A$ with $r_{A,h}(n)>0$ for all large $n$,
 \[ \limsup_{n\to\infty} \frac{r_{A,h}(n)}{\log n} > 0, \]
 so logarithmic growth is essentially best possible (see \cite{taf25} for a heuristic discussion).

 Probabilistic constructions were later extended to structured sets. Vu \cite{vvu00wp} showed that for each $k \geq 1$, there exist subsets $A \subseteq \mathbb{N}^k$ with $r_{A,h}(n) \asymp \log n$ for all sufficiently large $h \ge h_k = O(8^k k^3)$. Wooley \cite{woo03} and Pliego \cite{pli24} refined this, showing that the threshold $h_k$ essentially matches the asymptotic order of $\mathbb{N}^k$ obtained via the circle method. In our earlier work \cite{tafWWG}, we proved an analogous result for prime powers: for $h \geq k^2 -k + O(\sqrt{k})$, there exists $A \subseteq \mathbb{P}^k$ with $r_{A,h}(n) \asymp \log n$ for $n$ in a congruence class depending on $k$ and $h$.

 In this paper, we extend these results to Piatetski-Shapiro sequences, as well as $k$-th powers of Piatetski-Shapiro numbers and primes. For non-integer $c > 1$, set
 \[ \mathbb{N}_{(c)} := \{\lfloor n^c \rfloor ~|~ n \in \mathbb{N}\}, \qquad \mathbb{P}_{(c)} := \mathbb{N}_{(c)} \cap \mathbb{P}. \]
 The sequence $\mathbb{N}_{(c)}$ exhibits pseudorandom behavior for non-integer $c$, and has been studied extensively since Piatetski-Shapiro's theorem on the infinitude of $\mathbb{P}_{(c)}$ for $1 < c < 12/11$ \cite{pia53}. Our main results establish the existence of subbases of $\N_{(c)}$ (for all non-integer $c>1$) and of
 \[ \mathbb{N}_{(c)}^k := \{n^k ~|~ n \in \mathbb{N}_{(c)}\}, \qquad \mathbb{P}_{(c)}^k := \{p^k ~|~ p \in \mathbb{P}_{(c)}\} \]
 (for small $c>1$) with representation functions of prescribed regularly varying growth. The proofs combine probabilistic constructions from \cite{tafWWG} with a Hua-type mean value estimate for $\N_{(c)}$ by Poulias \cite{pou21}, together with exponential sum estimates for $\mathbb{N}_{(c)}^k$ and $\mathbb{P}_{(c)}^k$ due to Akbal--G\"ulo\u{g}lu \cite{akbgul16, akbgul18} and others.

\subsection{Piatetski-Shapiro sequences}
 Recall that a measurable function $F:(0,\infty)\to(0,\infty)$ is \emph{regularly varying} if $\lim_{x\to\infty} F(\lambda x)/F(x)$ exists for every $\lambda>0$.  Such functions take the form
 \[ F(x)=x^{\kappa}\psi(x), \]
 where $\kappa\in\R$ and $\psi$ is \emph{slowly varying} (i.e. $\psi(\lambda x)/\psi(x)\to 1$ as $x\to\infty$ for each $\lambda>0$); in particular $\psi(x)=x^{o(1)}$ (see Bingham--Goldie--Teugels \cite{bingham89}).
 
 In earlier work \cite{taf25, tafWWG}, we studied representation functions with growth prescribed by these functions. Specifically, for each $k \geq 1$ and $h \geq 2H_0(k)+1$, where
 \begin{equation}\label{h0k1}
  H_0(k) := \begin{cases}
          2^{k-1}, & 1\leq k\leq 4,\\
          \frac{1}{2}k(k-1)+\lfloor\sqrt{2k+2}\rfloor, & k\geq 5,
         \end{cases} 
 \end{equation}
 we proved that if $F$ is regularly varying and satisfies
 \[ \lim_{x\to\infty}\frac{F(x)}{\log x} = \infty \qquad \text{and} \qquad F(x)\leq (1+o(1)) \frac{\Gamma(1+1/k)^h}{\Gamma(h/k)} x^{h/k-1}, \]
 then there exists a set $A \subseteq \mathbb{N}^k$ such that
 \[ r_{A,h}(n) \sim \mathfrak{S}_{k,h}(n) F(n), \]
 where $\mathfrak{S}_{k,h}(n)$ is the singular series associated to Waring’s problem:
 \begin{equation}
  S(a,q) := \sum_{r=1}^{q} e\bigg(\frac{ar^k}{q}\bigg), \qquad \mathfrak{S}_{k,h}(n) := \sum_{q\geq 1} \sum_{\substack{a=1 \\ (a,q)=1}}^{q} \frac{S(a,q)^h}{q^h} \, e\bigg(-\frac{na}{q}\bigg). \label{singser}
 \end{equation}

 We will prove a similar result for $\N_{(c)}$. For non-integer $c>1$, let $H_0(c)$ be defined by
 \begin{equation}\label{h0c1}
  H_0(c) := \begin{cases}
          2, & 1< c < 2,\\
          \frac{1}{2}(\lfloor 2c\rfloor +1)(\lfloor 2c\rfloor + 2), & c>2.
         \end{cases} 
 \end{equation}
 Recent work of Madritsch \cite{mad26} treats Waring's problem for pseudo-polynomials and, in particular, obtains the minor arc estimates needed for sums involving $\lfloor n^c\rfloor$ for every non-integral $c>1$. More precisely, his method combines a Fourier approximation to the floor function with estimates for fractional-power exponential sums to bound 
 \[ \bigg|\sum_{n\leq x} e(\alpha\lfloor n^c\rfloor)\bigg| \]
 away from the central major arc. We use this input in Section \ref{sec-pss} to obtain the corresponding weighted minor arc bound for $\N_{(c)}$ (see Lemma \ref{pss-minor-sup}). Moreover, by results of Poulias \cite[Theorem 1.4]{pou21} and Robert--Sargos \cite[Theorem 2]{robsar06}, we have a Hua-type mean value estimate for Piatetski-Shapiro sequences:
 \[ \int_0^1 \bigg|\sum_{n\leq x} e(\alpha \lfloor n^c\rfloor)\bigg|^{s}\,\mathrm{d}\alpha \ll x^{s-c+o(1)} \qquad (s\geq 2H_0(c)). \]
 (See Proposition \ref{pss-hua}). 
 
 The main new point is that, despite the lack of an arithmetic singular series, one can still obtain a weighted asymptotic formula down to exponents slightly below $1/h$, which is the range required by the probabilistic subbasis criterion \cite[Theorem 1.4 and Corollary 1.6]{tafWWG}.
 
 \begin{thm}\label{MTcest}
  Let $c> 1$ be non-integral and $h\geq 2H_0(c)+1$, and let $\delta>0$ be any real number with 
  \[ \delta < \frac{h-2H_0(c)}{2 h(h-1) H_0(c)}. \]
  Then, for any $\omega \geq 1/h - \delta$,
  \begin{equation*}
   \sum_{\substack{x_1,\ldots,x_h \in \N_{(c)} \\ x_1+\cdots+x_h = N}} (x_1\cdots x_h)^{\omega-\frac{1}{c}} = \frac{1}{c^h}\frac{\Gamma(\omega)^h}{\Gamma(h\omega)} \, N^{h\omega-1} + O(N^{h\omega-1-\tau}),
  \end{equation*}
  for some $\tau = \tau(c,h,\delta) >0$.
 \end{thm}
 
 Combining this with Theorem \ref{main-tech}, we deduce:
 
 \begin{thm}\label{MTc}
  Let $c> 1$ be non-integral, and let $h\geq 2H_0(c)+1$. Let $F$ be a regularly varying function satisfying
  \[ \lim_{x\to\infty}\frac{F(x)}{\log x}=\infty,\qquad F(x)\leq (1+o(1))\frac{\Gamma(1+1/c)^h}{\Gamma(h/c)} x^{h/c-1}. \]
  Then there exists $A\subseteq \N_{(c)}$ such that $r_{A,h}(n)\sim F(n)$.
  
  Moreover, if $\log x \ll F(x) \ll x^{h/c-1}$, then there exists $A\subseteq \N_{(c)}$ with $r_{A,h}(n) \asymp F(n)$.
 \end{thm}

\subsection{Piatetski-Shapiro powers}
 We now turn to the $\N_{(c)}^k$ setting. Let $k \geq 2$ be an integer and $c > 1$ a non-integer. Assume that the exponential sum over $\N^k_{(c)}$ satisfies
 \begin{equation} \label{PSW}
  \sum_{n \leq x} \mathbbm{1}_{\N^k_{(c)}}(n)\, e(n\alpha) 
  = \frac{1}{c} \sum_{n \leq x} n^{(\frac{1}{c}-1)\frac{1}{k}}\, \mathbbm{1}_{\N^k}(n)\, e(n\alpha) + O(x^{\frac{1}{ck} - P(c,k)}),
 \end{equation}
 uniformly for $\alpha \in [0,1)$, where $0 < P(c,k) \leq 1/ck$ is a constant. This estimate provides the bridge that allows results for $\N^k$ to be extended to $\N^k_{(c)}$. Explicit ranges of $c$ and corresponding values of $P(c,k)$ for which \eqref{PSW} holds are given by Akbal--G\"ulo\u{g}lu \cite[Lemma 5]{akbgul16}:
 
 \begin{center}
  \renewcommand{\arraystretch}{1.35}
  \begin{tabular}{ccc}
   \toprule
   $k$ & range of $c$ & $P(c,k)$ \\ 
   \midrule
   $2$ & $\displaystyle 1<c<\frac{4}{3}$ & $\displaystyle \frac{1}{4c}-\frac{3}{16}$ \\[0.8em]
   $3$ & $\displaystyle 1<c<\frac{16}{15}$ & $\displaystyle \frac{8}{45c}-\frac{1}{6}$ \\[0.8em]
   $4$ & $\displaystyle 1<c<\frac{96}{95}$ & $\displaystyle \frac{96/c-95}{764}$ \\[0.8em]
   $5$ & $\displaystyle 1<c<\frac{224}{223}$ & $\displaystyle \frac{224/c-223}{2235}$ \\[0.8em]
   $k\geq 6$ & $\displaystyle 1<c<\frac{\nu_0(k)}{\nu_0(k)-1}$ & $\displaystyle \frac{\nu_0(k)/c-\nu_0(k)+1}{k(2\nu_0(k)-1)}$ \\
   \bottomrule
  \end{tabular}
 \end{center}
 
 \noindent
 where
 \[ \nu_0(k):= \begin{cases}
                \displaystyle \frac{3k+2}{k\,\sigma(3k/2)}, & \text{if } k\geq 6 \text{ is even}, \\[1.2em]
                \displaystyle \frac{3k+1}{(k-1)\,\sigma((3k-1)/2)}, & \text{if } k\geq 7 \text{ is odd},
               \end{cases} \]
 and $\sigma(r)^{-1} = r(r-1)$, as in Bourgain \cite[Theorem 5]{bou17}.
 
 From \eqref{PSW}, we obtain a direct analogue of Wooley’s \cite[Theorem 1.1]{woo04} for $\N_{(c)}^k$ (Theorem \ref{main-PSS1}). The methods of \cite{tafWWG} then lead to the following.

 \begin{thm}\label{MTPSSkappa}
  Suppose \eqref{PSW} holds, and let $H_0(k)$ be as in \eqref{h0k1}. Let
  \[ h \geq \max\{2H_0(k) + 1,\, 2\lceil \tfrac{1}{2} P(c,k)^{-1}\rceil+1\}. \] 
  Let $F$ be a regularly varying function satisfying
  \[ \lim_{x\to\infty}\frac{F(x)}{\log x}=\infty,\qquad F(x)\leq (1+o(1))\frac{\Gamma(1+1/ck)^h}{\Gamma(h/ck)} x^{h/ck-1}. \]
  Then there exists $A\subseteq \N_{(c)}^k$ such that $r_{A,h}(n)\sim \mathfrak{S}_{k,h}(n) F(n)$.
  
  Moreover, if $\log x \ll F(x) \ll x^{h/ck-1}$, then there exists $A\subseteq \N_{(c)}^k$ with $r_{A,h}(n) \asymp F(n)$.
 \end{thm}

%%%%%%%%%%%%%%%%%%%%%%%%%%%%%%%
\subsection{Piatetski-Shapiro prime powers}  
 Fix $k \geq 1$. For a prime $p$, let $\theta = \theta(k,p)$ denote the largest integer with $p^\theta \mid k$, and define
 \begin{equation} 
  \nu = \nu(k,p) := \begin{cases}
                     \theta + 2 & \text{if $p=2$ and $2\mid k$,} \\
                     \theta + 1 & \text{otherwise,}
                  \end{cases}
  \qquad
  K(k) := \prod_{(p-1)\mid k} p^\nu. \label{defkk}
 \end{equation}
 If $q$ is a prime coprime to $K(k)$, then $q^k \equiv 1 \pmod{p^\nu}$ whenever $(p-1)\mid k$, since $\varphi(p^\nu) \mid k$ for odd $p$ (and $\varphi(2^\nu)/2 \mid k$ when $p=2$), where $\varphi$ is Euler's totient function. By the Chinese remainder theorem, it follows that any $n$ representable as a sum of $h$ $k$-th powers of primes exceeding $k+1$ must satisfy $n \equiv h \pmod{K(k)}$. For example, $K(2) = 24$, while $K(k) = 2$ when $k$ is odd.

 In \cite{tafWWG} we proved that for each $h\geq 2H_0(k)+1$, where $H_0(k)$ is as in \eqref{h0k1}, and every regularly varying $F$ satisfying
 \[ \lim_{x\to\infty}\frac{F(x)}{\log x} = \infty \qquad \text{and} \qquad F(x)\leq (1+o(1)) \frac{\Gamma(1/k)^h}{\Gamma(h/k)} \frac{x^{h/k-1}}{(\log x)^{h}}, \]
 there exists a subset $A \subseteq \mathbb{P}^k$ such that
 \[ r_{A,h}(n) \sim \mathfrak{S}^*_{k,h}(n) F(n) \qquad (n \equiv h \pmod{K(k)}), \]
 where the singular series associated with the Waring--Goldbach problem is
 \begin{equation}
  S^{*}(q,a) := \sum_{\substack{r=1 \\ (r,q)=1}}^{q} e\bigg(\frac{ar^{k}}{q}\bigg),\qquad \mathfrak{S}^{*}_{k,h}(n) := \sum_{q\geq 1} \sum_{\substack{a=1 \\ (a,q)=1}}^{q} \frac{S^{*}(q,a)^h}{\varphi(q)^h}e\bigg(-\frac{na}{q}\bigg). \label{singserp}
 \end{equation}
 Moreover, for $\log x \ll F(x) \ll x^{h/k-1}/(\log x)^h$, there exists $A\subseteq \P^k$ with $r_{A,h}(n) \asymp F(n)$ ($n \equiv h \pmod{K(k)}$). In particular, there are thin subbases of prime powers.
 
 In the Piatetski-Shapiro setting, let $k \geq 1$ be an integer and $c > 1$ be non-integral. Assume that the exponential sum over $\P^k_{(c)}$ satisfies
 \begin{equation} \label{PSWG}
  \sum_{n \leq x} \mathbbm{1}_{\P^k_{(c)}}(n)\,(\log n)\, e(n\alpha) = \frac{1}{c} \sum_{n \leq x} n^{(\frac{1}{c}-1)\frac{1}{k}}\, \mathbbm{1}_{\P^k}(n)\,(\log n)\, e(n\alpha) + O(x^{\frac{1}{ck}-P^{*}(c,k)})
 \end{equation}
 uniformly for $\alpha\in[0,1)$, for some constant $0 < P^{*}(c,k) \leq 1/ck$. Variants of this estimate appear in the literature, for instance without the logarithmic weight (Akbal--G\"ulo\u{g}lu \cite[Lemma 2.11]{akbgul18}) or with $n^{(1-\frac{1}{c})\frac{1}{k}}$ on the left-hand side (Balog--Friedlander \cite[Theorem 4]{balfri92}). By standard partial summation, these versions are all equivalent to \eqref{PSWG} (see, e.g., Lemma \ref{PSWGequiv}).
 
 Explicit ranges of $c$ together with admissible values of $P^{*}(c,k)$ are given in the following table. More precisely, the quantity $P_0^*(c,k)$ displayed below is a threshold: because the quoted estimates contain harmless $x^{\eps}$ losses or logarithmic factors, any fixed $P^*(c,k)$ satisfying $0<P^*(c,k)<P_0^*(c,k)$ is admissible in \eqref{PSWG}:
 
 \begin{center}
  \small
  \setlength{\tabcolsep}{4pt}
  \renewcommand{\arraystretch}{1.35}
  \begin{tabular}{cccc}
    \toprule
    $k$ & range of $c$ & $P_0^*(c,k)$ & reference \\
    \midrule
    $1$ & $\displaystyle 1<c<\frac{73}{64}$ 
        & $\displaystyle \min\left\{1-\frac{1}{c},\,\frac{73/c-64}{86}\right\}$ 
        & Kumchev \cite[Theorem 2]{kum97} \\[0.8em]
    $2$ & $\displaystyle 1<c<\frac{82}{75}$ 
        & $\displaystyle \frac{82/c-75}{174}$ 
        & Zhang--Zhai \cite[Theorem 2]{zhazha05} \\[0.8em]
    $3$ & $\displaystyle 1<c<\frac{80}{77}$ 
        & $\displaystyle \min\left\{\frac{80/c-77}{468},\,\frac{78/c-75}{471}\right\}$ 
        & Akbal--G\"ulo\u{g}lu \cite[Lemma 2.11]{akbgul18} \\[0.8em]
    $k\geq 4$ & $\displaystyle 1<c<\frac{\nu_0^*(k)}{\nu_0^*(k)-1}$ 
        & $\displaystyle \frac{\nu_0^*(k)/c-\nu_0^*(k)+1}{k(2\nu_0^*(k)-1)}$ 
        & Akbal--G\"ulo\u{g}lu \cite[Lemma 2.11]{akbgul18} \\
    \bottomrule
  \end{tabular}
 \end{center}
 
 \noindent
 where
 \[ \nu_0^*(k):= \begin{cases}
                  \displaystyle k(k+1)^2, & \text{if } 4\leq k\leq 11, \\[0.8em]
                  \displaystyle \frac{2\lfloor 3k/2\rfloor\left(\lfloor 3k/2\rfloor^2-1\right)}
                  {\lfloor 3k/2\rfloor-k}, & \text{if } k\geq 12.
                \end{cases} \]
 
 With \eqref{PSWG}, we can prove an analogue of \cite[Theorem 1.2]{tafWWG} for $\P^k_{(c)}$ (Theorem \ref{main-PSP}). The methods of \cite{tafWWG} then allow us to deduce analogues of our earlier results for $\P^k_{(c)}$.

 \begin{thm}\label{MTPSPP}
  Suppose \eqref{PSWG} holds, and let $H_0(k)$ be as in \eqref{h0k1}. Let
  \[ h \geq \max\{2H_0(k) + 1,\, 2\lceil \tfrac{1}{2} P^{*}(c,k)^{-1}\rceil+1\}. \] 
  Let $F$ be a regularly varying function satisfying
  \[ \lim_{x\to\infty}\frac{F(x)}{\log x} = \infty,\qquad
     F(x)\leq (1+o(1))\frac{1}{c^h} \frac{\Gamma(1/ck)^h}{\Gamma(h/ck)} \frac{x^{h/ck-1}}{(\log x)^h}. \]
  Then there exists $A\subseteq \P_{(c)}^k$ such that $r_{A,h}(n)\sim \mathfrak{S}^{*}_{k,h}(n) F(n)$ \textnormal{($n \equiv h \pmod{K(k)}$)}, where $K(k)$ is as in \eqref{defkk}.
  
  Moreover, if $\log x \ll F(x) \ll x^{h/ck-1}/(\log x)^h$, then there exists $A\subseteq \P_{(c)}^k$ with $r_{A,h}(n) \asymp F(n)$ \textnormal{($n \equiv h \pmod{K(k)}$)}. 
 \end{thm}
 
 \begin{ntt}
  Throughout, $f(x)\ll g(x)$ (or $f(x)=O(g(x))$) means $|f(x)|\leq Cg(x)$ for some constant $C>0$, and $f(x)\asymp g(x)$ means both $f(x)\ll g(x)$ and $g(x)\ll f(x)$. Dependencies of implied constants on parameters are indicated by subscripts when relevant (usually omitted for $c$, $k$, $h$ and $f$). For $1\leq p<\infty$, we write
  \[ \|F\|_p := \bigg(\int_0^1 |F(\alpha)|^p\,\mathrm{d}\alpha\bigg)^{1/p}, \]
  and $\|F\|_\infty := \sup_{\alpha\in[0,1]}|F(\alpha)|$. Also, $e(\alpha) := e^{2\pi i\alpha}$.
 \end{ntt}

%%%%%%%%%%%%%%%%%%%%%%%%%%%%%%%% 
\section{Subbases with prescribed representation function}
 We start by recalling the main technical result of \cite{tafWWG}, which will be our tool for constructing thin subbases of Piatetski–Shapiro sets. Let $B\subseteq \N$ be a subset of the natural numbers, and write $B(x) := |B\cap [1,x]|$. Assume that
 \begin{equation}
  B(x) \asymp x^{\beta}\vartheta(x) \label{oreg}
 \end{equation}
 for some $\beta>0$ and some slowly varying function $\vartheta(x)$. In addition, suppose that $B$ satisfies the following Hua-type estimate:
 \begin{equation}\label{hua-type}
  \sum_{n\leq x} r_{B,H_0}(n)^2 \ll \frac{B(x)^{2H_0}}{x}\, x^{o(1)}
 \end{equation}
 for some $H_0=H_0(B)\geq 1$.
 
 The following is a direct consequence of \cite[Theorem 1.4 and Corollary 1.6]{tafWWG}.

 \begin{thm}[{\cite{tafWWG}}]\label{main-tech}
  Suppose $B\subseteq \N$ satisfies \eqref{oreg} and \eqref{hua-type}, and let $h \geq 2H_0+1$. Suppose that there exists some $\delta>0$ such that for every $1/h - \delta \leq \omega \leq \beta$, we have
  \begin{equation}
   \sum_{\substack{x_1,\ldots, x_h\in B \\ x_1+\cdots+ x_h = n}} \frac{x_1^{\omega}}{x_1^{\beta}\vartheta(x_1)}\cdots \frac{x_h^{\omega}}{x_h^{\beta}\vartheta(x_h)} \sim \mathfrak{S}(n)\, n^{h\omega - 1}, \label{mainest}
  \end{equation}
  where $\mathfrak{S}(n)$ is some function of $n$ (depending on $h,B,\omega$) satisfying 
  \[ \mathfrak{S}(n)\asymp 1 \quad \text{for } n\in \mathscr{S}, \]
  for some subset $\mathscr{S}\subseteq \N$.\footnote{In applications, the subset $\mathscr{S}$ will either be the full set of natural numbers (for $k$-th powers) or a congruence class of the form $a\cdot \N+b$ (for $k$-th powers of primes).} Then, for every regularly varying function $F(x) = x^{h\omega-1} \psi(x)$ with $F(n)\leq (1+o(1))B(n)^h/n$, we have the following:
  \begin{enumerate}[label=\textnormal{(\roman*)}]
   \item If $F(n)/\log n \to \infty$, there exists $A\subseteq B$ such that 
   \[ r_{A,h}(n) \sim \mathfrak{S}(n)F(n) \quad (n\in\mathscr{S}). \]
  
   \item If $F(n)\gg \log n$, there exists $A\subseteq B$ such that 
   \[ r_{A,h}(n) \asymp F(n) \quad (n\in\mathscr{S}). \]
  \end{enumerate}
 \end{thm}

 With this result, Theorems \ref{MTc}, \ref{MTPSSkappa}, and \ref{MTPSPP} follow directly from Propositions \ref{pss-hua}, \ref{PSShua} and \ref{PSPhua} (Hua-type bounds for $\mathbb{N}_{(c)}$, $\mathbb{N}^k_{(c)}$ and $\mathbb{P}^k_{(c)}$) together with the estimates in Theorems \ref{MTcest}, \ref{main-PSS1}, and \ref{main-PSP}. Bounds for the relevant singular series are discussed in \cite[Lemmas 4.5 and 5.6]{tafWWG}.

%%%%%%%%%%%%%%%%%%%%%%%%%% 
\section{Piatetski–Shapiro sequences}\label{sec-pss}
 In this section we prove Theorem \ref{MTcest}, from which Theorem \ref{MTc} will follow. In order to apply Theorem \ref{main-tech}, we first derive a Hua-type mean value estimate for $\N_{(c)}$, which follows from Robert--Sargos \cite[Theorem 2]{robsar06} (for $1<c<2$) and Poulias \cite[Theorem 1.4]{pou21} (for $c>2$). 
 
 \begin{prop}[Hua-type bound for $\N_{(c)}$]\label{pss-hua}
  Let $c>1$ be non-integral, and let $H_0(c)$ be as in \eqref{h0c1}. For every $h\geq H_0(c)$,
  \[ \sum_{n\leq x} r_{\N_{(c)},h}(n)^2 \ll x^{2h/c - 1 + o(1)}. \]
 \end{prop}
 \begin{proof}
  Let $P := (x+1)^{1/c}$, so that by Parseval, we have
  \[ \sum_{n\leq x} r_{\N_{(c)},h}(n)^2 \leq \int_0^1 \bigg|\sum_{n\leq P} e(\alpha \lfloor n^c\rfloor)\bigg|^{2h}\,\mathrm{d}\alpha. \]
  
  Assume first that $1<c<2$. By orthogonality, the integral $\int_{0}^{1} |\sum_{n\leq P} e(\alpha \lfloor n^c\rfloor)|^4\,\mathrm{d}\alpha$ counts the number of quadruples $n_1,n_2,n_3,n_4\leq P$ satisfying $\lfloor n_1^c\rfloor + \lfloor n_2^c\rfloor = \lfloor n_3^c\rfloor + \lfloor n_4^c\rfloor$. Every such quadruple satisfies
  \[ |n_1^c + n_2^c - n_3^c - n_4^c|\leq 2. \]
  Taking $\delta = 2P^{-c}$ in Theorem 2 of Robert--Sargos \cite{robsar06}, after a standard dyadic decomposition, the number of such quadruples is
  \[ \ll_\eps P^{2+\eps} + P^{4-c+\eps}. \]
  Since $c<2$, the first term is absorbed by the second. So for any $h\geq 2$, the trivial bound $|\sum_{n\leq P} e(\alpha \lfloor n^c\rfloor)|\leq P$ gives
  \[ \int_0^1 \bigg|\sum_{n\leq P} e(\alpha \lfloor n^c\rfloor)\bigg|^{2h}\,\mathrm{d}\alpha \leq P^{2h-4}\int_0^1 \bigg|\sum_{n\leq P} e(\alpha \lfloor n^c\rfloor)\bigg|^4\,\mathrm{d}\alpha \ll_\eps P^{2h-c+\eps}, \]
  as required.

  Now assume that $c>2$, and let $h\geq \frac{1}{2}(\lfloor 2c\rfloor + 1)(\lfloor 2c\rfloor + 2)$. Again by orthogonality, the integral $\int_0^1 |\sum_{n\leq P} e(\alpha \lfloor n^c\rfloor)|^{2h}\,\mathrm{d}\alpha$ counts the number of $m_1,\ldots,m_h,n_1,\ldots,n_h \in [1,P]\cap \Z$ such that $\sum_{i=1}^h \lfloor m_i^c\rfloor=\sum_{i=1}^h \lfloor n_i^c\rfloor$. Every such tuple satisfies
  \[ \bigg|\sum_{i=1}^h m_i^c-\sum_{i=1}^h n_i^c\bigg|\leq h. \]
  It is therefore enough to bound the number $J$ of such tuples for which the last inequality holds. Let $f(\beta) := \sum_{n\leq P} e(\beta n^c)$.
  We use the Davenport--Heilbronn kernel
  \[ K(\beta) = \bigg(\frac{\sin(2\pi h\beta)}{\pi\beta}\bigg)^2, \]
  which satisfies
  \[ \widehat{K}(y) := \int_{-\infty}^{\infty} K(\beta)\,e(-\beta y)\,\mathrm{d}\beta = (2h-|y|)_+. \]
  In particular, $\widehat K(y)\geq h$ whenever $|y|\leq h$, and hence $\mathbbm{1}_{[-h,h]}(y)\leq h^{-1}\widehat K(y)$. It follows that
  \begin{align*}
   J &\leq h^{-1}\sum_{\substack{1\leq m_1,\ldots,m_h \leq P \\ 1\leq n_1,\ldots,n_h \leq P}}\widehat K\bigg(\sum_{i=1}^h m_i^c-\sum_{i=1}^h n_i^c\bigg) \\
   &= h^{-1}\int_{-\infty}^{\infty} K(\beta) \sum_{\substack{1\leq m_1,\ldots,m_h \leq P \\ 1\leq n_1,\ldots,n_h \leq P}} e\bigg(-\beta\bigg(\sum_{i=1}^h m_i^c-\sum_{i=1}^h n_i^c \bigg)\bigg)\,\mathrm{d}\beta  \\
   &= h^{-1}\int_{-\infty}^{\infty} |f(\beta)|^{2h}K(\beta)\,\mathrm{d}\beta.
  \end{align*}
  By Theorem 1.4 of Poulias \cite{pou21}, since $h\geq \frac{1}{2}(\lfloor 2c\rfloor+1)(\lfloor 2c\rfloor+2)$, we have, for every $\kappa\geq 1$ and every $\eps>0$,
  \[ \int_{-\kappa}^{\kappa}|f(\beta)|^{2h}\,\mathrm{d}\beta \ll_{c,h,\eps}\kappa P^{2h-c+\eps}. \]
  Therefore, decomposing the integral against $K$ into the intervals $|\beta|\leq 1$ and $2^j<|\beta|\leq 2^{j+1}$, $j\geq 0$, and using that $K(\beta)\ll_h \min\{1,|\beta|^{-2}\}$, we obtain
  \begin{align*}
   \int_{-\infty}^{\infty}|f(\beta)|^{2h}K(\beta)\,\mathrm{d}\beta
   &\ll_h \int_{-1}^{1}|f(\beta)|^{2h}\,\mathrm{d}\beta + \sum_{j\geq 0} 2^{-2j}\int_{|\beta|\leq 2^{j+1}}|f(\beta)|^{2h}\,\mathrm{d}\beta  \\
   &\ll_{c,h,\eps} P^{2h-c+\eps} + \sum_{j\geq 0}2^{-2j}\, 2^{j+1}P^{2h-c+\eps}  \\
   &\ll_{c,h,\eps} P^{2h-c+\eps}.
  \end{align*}
  Hence
  \[ \int_0^1 \bigg|\sum_{n\leq P} e(\alpha \lfloor n^c\rfloor)\bigg|^{2h}\,\mathrm{d}\alpha \leq J \ll_{c,h,\eps} P^{2h-c+\eps}, \]
  giving the desired bound.
 \end{proof}
 
 Next, we derive asymptotics for the weighted solution counts needed to apply Theorem \ref{main-tech}.
 
\subsection{Setup for Theorem \ref{MTcest}}
 We now prove Theorem \ref{MTcest}. Fix a non-integral $c>1$, let $h\geq 2H_0(c)+1$, let $\omega\geq 1/h-\delta$, and let $N$ be a large integer. Put
 \[ T(\alpha;x) = \sum_{n\leq x} n^{\omega}\, \frac{\mathbbm{1}_{\N_{(c)}}(n)}{n^{1/c}}\, e(n\alpha),\qquad T^{\sharp}(\alpha;x) = \sum_{x/h\leq n\leq x} n^{\omega}\, \frac{\mathbbm{1}_{\N_{(c)}}(n)}{n^{1/c}}\, e(n\alpha). \]
 Thus, by orthogonality,
 \[ \sum_{\substack{x_1,\ldots,x_h\in \N_{(c)}\\ x_1+\cdots+x_h=N}} (x_1\cdots x_h)^{\omega-\frac{1}{c}} =
 \int_0^1 T(\alpha;N)^h e(-N\alpha)\,\mathrm{d}\alpha. \]
 If $x_1+\cdots+x_h=N$, then $x_i\geq N/h$ for at least one $i$. Hence
 \[ \int_0^1 (T(\alpha;N)-T^{\sharp}(\alpha;N))^h e(-N\alpha)\,\mathrm{d}\alpha=0. \]
 Therefore
 \begin{align}
  \sum_{\substack{x_1,\ldots,x_h\in \N_{(c)}\\ x_1+\cdots+x_h=N}} (x_1\cdots x_h)^{\omega-\frac{1}{c}} &=
  \int_0^1\Big(T(\alpha;N)^h-(T(\alpha;N)-T^{\sharp}(\alpha;N))^h\Big)e(-N\alpha)\,\mathrm{d}\alpha \nonumber \\
  &= \sum_{j=1}^h(-1)^{j+1}\binom{h}{j} \int_0^1 T^{\sharp}(\alpha;N)^jT(\alpha;N)^{h-j}e(-N\alpha)\,\mathrm{d}\alpha. \label{pss-expand}
 \end{align}

 Let $\omega_0 := 1/h-\delta$. We have $\omega_0>0$ by the hypothesis on $\delta$. Choose $\nu>0$ sufficiently small so that
 \[ \nu < \frac{1}{3} \min\bigg\{\frac{1}{c},\,\omega_0\bigg\}. \]
 We define the major arc by
 \[ \mathfrak{M}:=\left\{\alpha\in[0,1] ~\bigg|~ \|\alpha\|\leq \frac{L}{N}\right\},\qquad L:=N^\nu, \]
 where $\|\cdot\|$ denotes distance to the nearest integer, and put $\mathfrak{m}:=[0,1]\setminus\mathfrak{M}$. We shall prove a major arc asymptotic (Proposition \ref{major-pss}) and a minor arc estimate with power saving (Proposition \ref{minor-pss}), from which Theorem \ref{MTcest} will follow at once from \eqref{pss-expand}.

\subsection{Major arcs}
 In this subsection, we will prove the following:
 
 \begin{prop}[Major arcs]\label{major-pss}
  There exists $\eta>0$ such that
  \begin{align*}
   \sum_{j=1}^h(-1)^{j+1}\binom{h}{j} \int_{\mathfrak{M}} T^{\sharp}(\alpha;N)^j T(\alpha;N)^{h-j} e(-N\alpha)\,\mathrm{d}\alpha = \frac{1}{c^h}\frac{\Gamma(\omega)^h}{\Gamma(h\omega)} N^{h\omega-1} + O(N^{h\omega-1-\eta}).
  \end{align*}
 \end{prop}

 In order to calculate the integral, we first approximate $T$ and $T^{\sharp}$ on $\mathfrak{M}$.

 \begin{lem}\label{pss-major-approx}
  Let
  \[ U(\alpha;x) := \sum_{n\leq x} n^{\omega-1} e(n\alpha),\qquad U^{\sharp}(\alpha;x) := \sum_{x/h\leq n\leq x} n^{\omega-1} e(n\alpha). \]
  Then, uniformly for $\alpha\in\mathfrak{M}$,
  \[ T(\alpha;N) = \frac{1}{c} U(\alpha;N) + O(N^{\omega-2\nu}),\qquad T^{\sharp}(\alpha;N) = \frac{1}{c} U^{\sharp}(\alpha;N) + O(N^{\omega-2\nu}). \]
 \end{lem}
 \begin{proof}
  We prove the estimate for $T$; the proof for $T^{\sharp}$ follows by subtraction. Let
  \[ a(n) := n^{\omega-\frac{1}{c}}\mathbbm{1}_{\N_{(c)}}(n) - \frac{1}{c} n^{\omega-1}, \qquad A(t) := \sum_{n\leq t} a(n). \]
  We claim that
  \[ A(t) \ll 1 + t^{\omega-\frac{1}{c}}. \]
  Indeed, since $\lfloor m^c\rfloor = m^c + O(1)$,
  \begin{align*}
   \sum_{n\leq t}n^{\omega-\frac{1}{c}}\mathbbm{1}_{\N_{(c)}}(n) &= \sum_{m\leq t^{1/c}} \lfloor m^c\rfloor^{\omega-\frac{1}{c}} + O(t^{\omega-\frac{1}{c}}) \\
   &= \sum_{m\leq t^{1/c}} m^{c\omega-1} + O(1 + t^{\omega-\frac{1}{c}}) \\
   &= \frac{t^\omega}{c\omega} + O(1 + t^{\omega-\frac{1}{c}}).
  \end{align*}
  On the other hand, $\frac{1}{c} \sum_{n\leq t} n^{\omega-1} = t^\omega/c\omega + O(1 + t^{\omega-1})$, which proves the claim.

  Hence, by partial summation, since $|\alpha|\leq L/N=N^{-1+\nu}$ on $\mathfrak{M}$,\footnote{We may assume $-\frac{1}{2}<\alpha\leq \frac{1}{2}$.}
  \begin{align*}
   T(\alpha;N)-\frac1cU(\alpha;N) &= \sum_{n\leq N}a(n)e(n\alpha) \\
   &= A(N)e(N\alpha)-2\pi i\alpha\int_1^N A(t)e(t\alpha)\,\mathrm{d}t \\
   &\ll 1 + N^{\omega-\frac{1}{c}} + N^{\nu}(1 + N^{\omega-\frac{1}{c}}).
  \end{align*}
  If $\omega\geq 1/c$, this is $O(N^{\omega-\frac{1}{c}+\nu})$, while if $\omega < 1/c$, it is $O(N^{\nu})$. Since $\omega \geq \omega_0$ and $3\nu < \min\{1/c,\omega_0\}$, this is $O(N^{\omega-2\nu})$ uniformly in the allowed range of $\omega$.
 \end{proof}

 We are now able to substitute $T,T^{\sharp}$ by $U,U^{\sharp}$ in the integral of Proposition \ref{major-pss}.

 \begin{proof}[Proof of Proposition \ref{major-pss}]
  By Lemma \ref{pss-major-approx} and the trivial bounds 
  \[ |T(\alpha;N)|,|T^{\sharp}(\alpha;N)|,|U(\alpha;N)|,|U^\sharp(\alpha;N)|\ll N^\omega, \]
  we have, uniformly for $\alpha\in\mathfrak M$,
  \[ T^{\sharp}(\alpha;N)^j T(\alpha;N)^{h-j} = \frac{1}{c^h} U^{\sharp}(\alpha;N)^j U(\alpha;N)^{h-j} + O(N^{h\omega-2\nu}). \]
  Since $\int_{\mathfrak M} \mathrm{d}\alpha = 2N^{-1+\nu}$, this gives
  \begin{align}
   &\sum_{j=1}^h(-1)^{j+1}\binom{h}{j}
   \int_{\mathfrak{M}}T^{\sharp}(\alpha;N)^jT(\alpha;N)^{h-j}e(-N\alpha)\,\mathrm{d}\alpha \nonumber\\
   &\quad = \frac{1}{c^h} \sum_{j=1}^h(-1)^{j+1}\binom{h}{j} \int_{\mathfrak{M}} U^{\sharp}(\alpha;N)^j U(\alpha;N)^{h-j} e(-N\alpha)\,\mathrm{d}\alpha
   +O(N^{h\omega-1-\nu}). \label{pss-major-replace}
  \end{align}

  We now extend the integral from $\mathfrak{M}$ to $[0,1]$. For $0<\|\alpha\|\leq 1/2$, the estimates of \cite[Lemma 4.6]{tafWWG} give
  \[ |U(\alpha;N)|\ll N^{\omega-1}\|\alpha\|^{-1}+\|\alpha\|^{-\omega}, \qquad |U^\sharp(\alpha;N)|\ll N^{\omega-1}\|\alpha\|^{-1}. \]
  Therefore, exactly as in \cite[Eq. (4.9)]{tafWWG}, for $1\leq j\leq h$,
  \[ \int_{[0,1]\setminus\mathfrak M}|U^\sharp(\alpha;N)|^j|U(\alpha;N)|^{h-j}\,\mathrm{d}\alpha \ll \frac{N^{h\omega-1}}{L^{(h-1)\min\{1,\omega\}}} \ll N^{h\omega-1-\eta_1} \]
  for some $\eta_1 > 0$. Hence the right-hand side of \eqref{pss-major-replace} equals
  \[ \frac{1}{c^h} \sum_{j=1}^h(-1)^{j+1}\binom{h}{j} \int_0^1 U^{\sharp}(\alpha;N)^j U(\alpha;N)^{h-j} e(-N\alpha)\,\mathrm{d}\alpha + O(N^{h\omega-1-\eta_2}) \]
  with some $\eta_2>0$.

  Finally, by the same argument used to obtain \eqref{pss-expand}, applied to $U$ and $U^\sharp$, we have
  \begin{align*}
   \frac{1}{c^h} \sum_{j=1}^h(-1)^{j+1}\binom{h}{j} \int_0^1U^\sharp(\alpha;N)^j U(\alpha;N)^{h-j} e(-N\alpha)\,\mathrm{d}\alpha &= \frac{1}{c^h}\int_0^1 U(\alpha;N)^h e(-N\alpha)\,\mathrm{d}\alpha \\
   &= \frac{1}{c^h}\sum_{\substack{n_1,\ldots,n_h\in\N\\ n_1+\cdots+n_h=N}} (n_1\cdots n_h)^{\omega-1}.
  \end{align*}
  By the standard singular integral estimate (cf. \cite[Lemma 3.5]{taf25}),
  \[ \frac{1}{c^h}\sum_{\substack{n_1,\ldots,n_h\in\N\\ n_1+\cdots+n_h=N}} (n_1\cdots n_h)^{\omega-1} = \frac{1}{c^h}\frac{\Gamma(\omega)^h}{\Gamma(h\omega)}N^{h\omega-1} + O(N^{h\omega-1-\eta_3}) \]
  for some $\eta_3>0$. This proves the proposition.
 \end{proof}

%%%%%%%%%%%%%%%%%%%%%%%% 
\subsection{Minor arcs}
 We now turn to the minor arcs $\mathfrak{m}$. Together with \eqref{pss-expand} and Proposition \ref{major-pss}, the next result directly implies Theorem \ref{MTcest}.

 \begin{prop}[Minor arcs]\label{minor-pss}
  There exists $\eta'>0$ such that
  \[ \sum_{j=1}^h \bigg| \int_{\mathfrak{m}} T^{\sharp}(\alpha;N)^j T(\alpha;N)^{h-j} e(-N\alpha)\,\mathrm{d}\alpha \bigg| \ll N^{h\omega-1-\eta'}. \]
 \end{prop}

 To bound the value of $T^{\sharp}$ for $\alpha\in\mathfrak{m}$, we apply a result of Madritsch \cite[Sections 5--7]{mad26}.
 
 \begin{lem}\label{pss-minor-sup}
  There exists $\lambda=\lambda(c,\nu)>0$ such that
  \[ \sup_{\alpha\in\mathfrak{m}} |T^{\sharp}(\alpha;N)|\ll N^{\omega-\lambda+o(1)}. \]
 \end{lem}
 \begin{proof}
  Recall that $\nu < \frac{1}{3}\min\{\frac{1}{c},\omega_0\}$. Put $P=(N+1)^{1/c}$ and choose
  \[ 0 < \xi < \min\{c\nu,c-1,\tfrac{1}{5}\}. \]
  We use the following minor arc estimate of Madritsch \cite[Sections 5--7, especially (6.15) and (7.1)]{mad26}, specialized to $f(x)=x^c$: for every $\eps>0$,
  \[ \sup_{\|\alpha\|\geq P^{-c+\xi}} \bigg|\sum_{m\leq P} e(\alpha\lfloor m^c\rfloor)\bigg|
   \ll_{c,\xi,\eps} P^{1-\sigma+\eps}, \qquad \sigma := \frac{\xi}{2\lceil c\rceil(\lceil c\rceil+1)}. \]
  By subtraction, the same estimate holds uniformly for sums over intervals $M_1<m\leq M_2\leq P$.

  If $\alpha\in\mathfrak{m}$, then $\|\alpha\|\geq N^{-1+\nu}$. Since $\xi < c\nu$, it follows that
  \[ \|\alpha\|\geq N^{-1+\nu}\gg P^{-c+c\nu}\geq P^{-c+\xi}. \]
  Hence, uniformly for $\alpha\in\mathfrak{m}$ and for intervals $M_1<m\leq M_2\leq P$, we have
  \[ \sum_{M_1<m\leq M_2}e(\alpha\lfloor m^c\rfloor) \ll P^{1-\sigma+o(1)}. \]
  Since $n\in\N_{(c)}$ if and only if $n=\lfloor m^c\rfloor$ for some $m$, we may write
  \[ T^{\sharp}(\alpha;N) = \sum_{(N/h)^{1/c}\leq m\leq N^{1/c}} \lfloor m^c\rfloor^{\omega-\frac{1}{c}} e(\alpha\lfloor m^c\rfloor) + O(N^{\omega-\frac{1}{c}}). \]
  On this range, $\lfloor m^c\rfloor\asymp N$, and the weights are monotone. Therefore, by partial summation,
  \[ T^{\sharp}(\alpha;N) \ll N^{\omega-\frac1c}P^{1-\sigma+o(1)} = N^{\omega-\frac{\sigma}{c}+o(1)}. \]
  Thus the lemma follows by taking any fixed
  \[ 0<\lambda<\frac{\sigma}{c} = \frac{\xi}{2c\lceil c\rceil(\lceil c\rceil+1)}. \qedhere \]
 \end{proof}

 Next, we need an auxiliary lemma concerning the integrals of $T$ and $T^{\sharp}$.

 \begin{lem}\label{pss-weighted-mv}
  For every $\ell\geq 2H_0(c)$, we have
  \[ \int_0^1 |T(\alpha;x)|^\ell\,\mathrm{d}\alpha \ll x^{\max\{\ell\omega-1,0\}+o(1)},\qquad \int_0^1 |T^{\sharp}(\alpha;x)|^\ell\,\mathrm{d}\alpha \ll x^{\ell\omega-1+o(1)}. \]
 \end{lem}
 \begin{proof}
  Put $g(\alpha;t):=\sum_{n\leq t} \mathbbm{1}_{\N_{(c)}}(n)\,e(n\alpha)$. Since trivially $|g(\alpha;t)|\ll t^{1/c}$, it follows by Proposition \ref{pss-hua} and orthogonality that, for every $\ell\geq 2H_0(c)$,
  \[ \int_0^1 |g(\alpha;t)|^\ell\,\mathrm{d}\alpha \ll t^{(\ell-2H_0(c))/c} \int_0^1 |g(\alpha;t)|^{2H_0(c)}\,\mathrm{d}\alpha \ll t^{\ell/c-1+o(1)}. \]
  Equivalently, $\|g(\cdot;t)\|_\ell\ll t^{1/c-1/\ell+o(1)}$. By partial summation,
  \[ T(\alpha;x) = x^{\omega-\frac{1}{c}}g(\alpha;x) - \int_1^x g(\alpha;t)\,\mathrm{d}(t^{\omega-\frac{1}{c}}). \]
  Hence, by the triangle inequality
  \begin{align*}
   \|T(\cdot;x)\|_\ell &\ll x^{\omega-\frac{1}{c}} \|g(\cdot;x)\|_\ell + \int_1^x t^{\omega-\frac{1}{c}-1}\|g(\cdot;t)\|_\ell\,\mathrm{d}t \\
   &\ll x^{\omega-\frac{1}{\ell}+o(1)} + \int_1^x t^{\omega-\frac{1}{\ell}-1+o(1)}\,\mathrm{d}t \\
   &\ll x^{\max\{\omega-\frac{1}{\ell},0\}+o(1)},
  \end{align*}
  and raising to the power $\ell$ gives the first estimate.

  For $T^{\sharp}$, put $y := \lceil x/h\rceil - 1$. Partial summation gives
  \[ T^{\sharp}(\alpha;x) = x^{\omega-\frac{1}{c}}g(\alpha;x) - y^{\omega-\frac{1}{c}} g(\alpha;y) -\int_{y}^x g(\alpha;t)\,\mathrm{d}(t^{\omega-\frac{1}{c}}). \]
  Hence
  \begin{align*}
   \|T^{\sharp}(\cdot;x)\|_\ell &\ll x^{\omega-\frac{1}{c}} (x^{\frac{1}{c}-\frac{1}{\ell}+o(1)}) + \int_{\lceil x/h\rceil -1}^x t^{\omega-\frac{1}{c}-1} (t^{\frac{1}{c}-\frac{1}{\ell}+o(1)})\,\mathrm{d}t  \\
   &\ll x^{\omega-\frac{1}{\ell}+o(1)},
  \end{align*}
  and raising to the power $\ell$ finishes the proof.
 \end{proof}
 
 We are now ready to prove Proposition \ref{minor-pss}.
 
 \begin{proof}[Proof of Proposition \ref{minor-pss}]
  We first treat the case $\omega\geq 1/h$. By Lemmas \ref{pss-minor-sup} and \ref{pss-weighted-mv},
  \begin{align*}
   \int_{\mathfrak{m}}|T^{\sharp}(\alpha;N)|^h\,\mathrm{d}\alpha &\leq \bigg(\sup_{\alpha\in\mathfrak{m}}| T^{\sharp}(\alpha;N)|\bigg)^{h-2H_0(c)} \int_0^1 |T^{\sharp}(\alpha;N)|^{2H_0(c)}\,\mathrm{d}\alpha \\
   &\ll (N^{\omega-\lambda+o(1)})^{h-2H_0(c)} N^{2H_0(c)\omega-1+o(1)} \\
   &\ll N^{h\omega-1-\lambda+o(1)}.
  \end{align*}
  Since $\omega\geq 1/h$, Lemma \ref{pss-weighted-mv} gives
  \[ \int_0^1|T(\alpha;N)|^h\,\mathrm{d}\alpha\ll N^{h\omega-1+o(1)}. \]
  Therefore, by H\"older's inequality,
  \begin{align*}
   &\sum_{j=1}^h \bigg| \int_{\mathfrak{m}} T^{\sharp}(\alpha;N)^j T(\alpha;N)^{h-j} e(-N\alpha)\,\mathrm{d}\alpha \bigg| \\
   &\hspace{+12em} \leq \sum_{j=1}^h \bigg(\int_{\mathfrak{m}}|T^{\sharp}(\alpha;N)|^h\,\mathrm{d}\alpha\bigg)^{j/h} \bigg(\int_0^1|T(\alpha;N)|^h\,\mathrm{d}\alpha\bigg)^{1-j/h} \\
   &\hspace{+12em}\ll N^{h\omega-1- \lambda/h + o(1)},
  \end{align*}
  thus giving the required power saving.

  Now suppose that $1/h-\delta\leq \omega<1/h$, and put $\theta:=1-(h-1)\omega$. Then
  \[ \theta \leq \frac{1}{h} + (h-1)\delta < \frac{1}{2H_0(c)}. \]
  For $1\leq j\leq h$, H\"older's inequality (cf. Wooley \cite[Eq. (2.7)]{woo04}) gives
  \begin{equation}
   \bigg| \int_{\mathfrak{m}}T^{\sharp}(\alpha;N)^j T(\alpha;N)^{h-j} e(-N\alpha)\,\mathrm{d}\alpha \bigg| \ll \bigg(\sup_{\alpha\in\mathfrak{m}}|T^{\sharp}(\alpha;N)|\bigg)^{1-2H_0(c)\theta} \Upsilon_1^\theta \Upsilon_2^{(j-1)\omega} \Upsilon_3^{(h-j)\omega}, \label{pss-holder2}
  \end{equation}
  where
  \[ \Upsilon_1:=\int_0^1|T^{\sharp}(\alpha;N)|^{2H_0(c)}\,\mathrm{d}\alpha,\quad \Upsilon_2:=\int_0^1|T^{\sharp}(\alpha;N)|^{1/\omega}\,\mathrm{d}\alpha,\quad \Upsilon_3:=\int_0^1|T(\alpha;N)|^{1/\omega}\,\mathrm{d}\alpha. \]
  By Lemma \ref{pss-minor-sup},
  \[ \sup_{\alpha\in\mathfrak{m}}|T^{\sharp}(\alpha;N)|\ll N^{\omega-\lambda+o(1)}. \]
  Moreover, Lemma \ref{pss-weighted-mv} gives
  \[ \Upsilon_1\ll N^{2H_0(c)\omega-1+o(1)}. \]
  Since $\omega<1/h$, we have $1/\omega>h>2H_0(c)$, and so Lemma \ref{pss-weighted-mv} also gives
  \[ \Upsilon_2,\Upsilon_3\ll N^{o(1)}. \]
  Substituting these estimates into \eqref{pss-holder2}, we obtain
  \[ \bigg| \int_{\mathfrak{m}}T^{\sharp}(\alpha;N)^j T(\alpha;N)^{h-j}e(-N\alpha)\,\mathrm{d}\alpha \bigg| \ll N^{h\omega-1-\lambda(1-2H_0(c)\theta)+o(1)}. \]
  Finally,
  \[ 1-2H_0(c)\theta \geq \frac{h-2H_0(c)}{h}-2H_0(c)(h-1)\delta > 0 \]
  by the bound on $\delta$. Thus the minor arc contribution is $O(N^{h\omega-1-\eta'})$ for some $\eta'>0$.
 \end{proof}

\subsection{Proof of Theorem \ref{MTc}}
 Since $|\N_{(c)}\cap[1,x]| = x^{1/c} + O(1)$, the set $\N_{(c)}$ has regularly varying counting function. By Proposition \ref{pss-hua}, it also satisfies \eqref{hua-type} with $H_0(\N_{(c)}) = H_0(c)$. Theorem \ref{MTcest} gives \eqref{mainest} with $\mathscr{S}=\N$ and $\mathfrak{S}(n)$ a constant, for $h\geq 2H_0(c)+1$. Hence, Theorem \ref{main-tech} yields Theorem \ref{MTc}. \hfill$\square$

%%%%%%%%%%%%%%%%%%%%%%%%%%%%%%%%%%
\section{Piatetski-Shapiro powers}\label{sec-psw}
 In this section we prove Theorem \ref{MTPSSkappa}. We assume throughout that $\N^k_{(c)}$ satisfies \eqref{PSW}. We start by establishing a weighted Hua-type estimate for $\N^{k}$, from which a Hua-type estimate for $\N_{(c)}^k$ (Proposition \ref{PSShua}) will be derived.
 
 \begin{lem}\label{whuaPPow}
  Let $H_0 = H_0(k)$ be as in \eqref{h0k1}. For every $\ell \geq 1$, we have
  \[ \int_0^1 \bigg| \sum_{n \leq x} n^{\omega}\, \frac{\mathbbm{1}_{\N^k}(n)}{n^{1/k}}\,e(n\alpha) \bigg|^{\ell} \, \mathrm{d}\alpha \ll x^{\max\{\ell\omega - \frac{\ell}{\ell^{*}},\,0\} + o(1)}, \]
  where $\ell^{*} := \max\{\ell,2H_0\}$. Moreover, for any $C\geq 1$,
  \[ \int_0^1 \bigg| \sum_{x/C \leq n \leq x} n^{\omega}\, \frac{\mathbbm{1}_{\N^k}(n)}{n^{1/k}}\,e(n\alpha) \bigg|^{\ell} \, \mathrm{d}\alpha \ll_C x^{\ell\omega - \frac{\ell}{\ell^{*}} + o(1)}. \]
 \end{lem}
 \begin{proof}
  As shown by Wooley \cite[Corollary 14.7]{woo19}, we have
  \begin{equation}
   \sum_{n\leq x} r_{\N^{k},h}(n)^{2} \leq \int_{0}^{1} \bigg|\sum_{n\leq x} \mathbbm{1}_{\N^k}(n)\, e(\alpha n)\bigg|^{2h}\,\mathrm{d}\alpha \ll x^{\frac{2h}{k}-1 + o(1)} \quad (\forall h\geq H_0(k)). \label{Phua}
  \end{equation}
  Let $g(\alpha;x):=\sum_{n\leq x}\mathbbm{1}_{\N^k}(n)\,e(n\alpha)$.   We claim that, uniformly for $1\leq t\leq x$,
  \[ \|g(\cdot;t)\|_{\ell}\ll t^{\frac{1}{k}-\frac{1}{\ell^*}+o(1)}. \]
  Indeed, if $\ell\leq 2H_0$, then $\ell^*=2H_0$, and the claim follows from \eqref{Phua} and the monotonicity of $L^p$-norms. Now suppose that $\ell>2H_0$, so that $\ell^*=\ell$. By the trivial bound $\|g(\cdot;t)\|_{\infty}\leq |\N^k\cap[1,t]|\ll t^{1/k}$ and \eqref{Phua} with $h=H_0$, we have
  \[ \|g(\cdot;t)\|_{\ell}^{\ell}\leq \|g(\cdot;t)\|_{\infty}^{\ell-2H_0}\|g(\cdot;t)\|_{2H_0}^{2H_0}\ll t^{\frac{\ell-2H_0}{k}}t^{\frac{2H_0}{k}-1+o(1)}=t^{\frac{\ell}{k}-1+o(1)}. \]
  Taking $\ell$-th roots gives $\|g(\cdot;t)\|_{\ell}\ll t^{\frac{1}{k}-\frac{1}{\ell}+o(1)}$.

  By partial summation,
  \[ \sum_{n\leq x} n^{\omega}\,\frac{\mathbbm{1}_{\N^k}(n)}{n^{1/k}}\,e(n\alpha) = x^{\omega-\frac{1}{k}} g(\alpha;x)-(\omega-\tfrac{1}{k})\int_1^x t^{\omega-\frac{1}{k}-1}g(\alpha;t)\,\mathrm{d}t. \]
  Hence, by the triangle inequality,
  \begin{align*}
   \bigg\|\sum_{n\leq x} n^{\omega}\,\frac{\mathbbm{1}_{\N^k}(n)}{n^{1/k}}\,e(n\alpha)\bigg\|_{\ell}
   &\ll x^{\omega-\frac{1}{k}}\|g(\cdot;x)\|_{\ell}+\int_1^x t^{\omega-\frac{1}{k}-1}\|g(\cdot;t)\|_{\ell}\,\mathrm{d}t \\
   &\ll x^{\omega-\frac{1}{\ell^*}+o(1)}+\int_1^x t^{\omega-\frac{1}{\ell^*}-1+o(1)}\,\mathrm{d}t.
  \end{align*}
  If $\omega>1/\ell^*$, the last expression is $\ll x^{\omega-\frac{1}{\ell^*}+o(1)}$. If $\omega=1/\ell^*$, it is $\ll x^{o(1)}$. If $\omega<1/\ell^*$, it is $\ll 1$. Raising to the power $\ell$ gives the result.

  For the short interval estimate, put $y:=\lceil x/C\rceil-1$. By partial summation,
  \begin{align*}
   \sum_{x/C\leq n\leq x} n^{\omega}\,\frac{\mathbbm{1}_{\N^k}(n)}{n^{1/k}}\,e(n\alpha)
   &= x^{\omega-\frac{1}{k}} g(\alpha;x) - y^{\omega-\frac{1}{k}} g(\alpha;y)-(\omega-\tfrac{1}{k})\int_y^x t^{\omega-\frac{1}{k}-1}g(\alpha;t)\,\mathrm{d}t.
  \end{align*}
  Since $y\asymp_C x$, the same argument gives
  \begin{align*}
   \bigg\|\sum_{x/C\leq n\leq x} n^{\omega}\,\frac{\mathbbm{1}_{\N^k}(n)}{n^{1/k}}\,e(n\alpha)\bigg\|_{\ell}
   &\ll_C x^{\omega-\frac{1}{\ell^*}+o(1)}+\int_{\lceil x/C\rceil - 1}^{x} t^{\omega-\frac{1}{\ell^*}-1+o(1)}\,\mathrm{d}t \\
   &\ll_C x^{\omega-\frac{1}{\ell^*}+o(1)}.
  \end{align*}
  Raising to the power $\ell$ finishes the proof.
 \end{proof}

 \begin{prop}[Hua-type bound for $\N_{(c)}^{k}$]\label{PSShua}
  Suppose \eqref{PSW} holds, and let $H_0(k)$ be as in \eqref{h0k1}. For every $h\geq \max\{H_0(k), \tfrac{1}{2} P(c,k)^{-1}\}$,
  \[ \sum_{n\leq x} r_{\N^k_{(c)},h}(n)^2 \ll x^{\frac{2h}{ck}-1+o(1)}. \]
 \end{prop}
 \begin{proof}
  Lemma \ref{whuaPPow} with $\ell=2h$, $\omega = 1/ck$ gives
  \[ \int_0^1 \bigg|\frac{1}{c}\sum_{n\leq x} n^{(\frac{1}{c}-1)\frac{1}{k}}\mathbbm{1}_{\N^k}(n)\,e(\alpha n)\bigg|^{2h}\,d\alpha \ll x^{\frac{2h}{ck}-1+o(1)}. \]
  Raising both sides of \eqref{PSW} to the $2h$-th power and integrating, the inequality $(x+y)^{2h}\ll_h x^{2h}+y^{2h}$ gives
  \begin{align*}
   \sum_{n\leq x} r_{\N^k_{(c)},h}(n)^2 &\leq \int_0^1 \bigg|\sum_{n\leq x}\mathbbm{1}_{\N^k_{(c)}}(n)\,e(\alpha n)\bigg|^{2h}\,d\alpha \\
   &\ll \int_0^1 \bigg|\frac{1}{c}\sum_{n\leq x} n^{(\frac{1}{c}-1)\frac{1}{k}}\mathbbm{1}_{\N^k}(n)\,e(\alpha n)\bigg|^{2h}\,d\alpha + \int_0^1 O(x^{2h(\frac{1}{ck}-P(c,k))})\,d\alpha \\
   &\ll x^{\frac{2h}{ck}-1+o(1)} + x^{\frac{2h}{ck}- 2hP(c,k)}.
  \end{align*}
  For $h\geq \tfrac{1}{2} P(c,k)^{-1}$ the second term is absorbed by the first, and the proposition follows.
 \end{proof}
 
 To apply Theorem \ref{main-tech}, we will need the following, direct analogue of \cite[Theorem 4.2]{tafWWG} for $\N_{(c)}^k$, which is based on Wooley \cite[Theorem 1.1]{woo04}:

 \begin{thm}\label{main-PSS1}
  Suppose \eqref{PSW} holds, let $H_0 = H_0(k)$ be as in \eqref{h0k1}, and let $h\geq \max\{2H_0+1, P(c,k)^{-1}+1\}$. Let $\delta>0$ be any real number with
  \[ \delta < \frac{1}{h(h-1)^2}. \]
  Then, for any $\omega\geq 1/h - \delta$,
  \[ \sum_{\substack{x_1,\ldots,x_h\in \N^k_{(c)} \\ x_1+\cdots+x_h = N}} (x_1\cdots x_h)^{\omega - \frac{1}{ck}} = \mathfrak{S}_{k,h}(N)\frac{1}{(ck)^h}\frac{\Gamma(\omega)^h}{\Gamma(h\omega)} N^{h\omega - 1} + O(N^{h\omega -1 -\nu}), \]
  for some $\nu = \nu(k,h,\delta)>0$, where the singular series $\mathfrak{S}_{k,h}(N)$ is defined as in \eqref{singser}.
 \end{thm}
 
 The proof of Theorem \ref{main-PSS1} will be based on \eqref{PSW}, which we first generalize to weighted sums. 
 
 \begin{lem}\label{trnsfr-PSS}
  Suppose \eqref{PSW} holds, and let $\omega \in \R$. Then, uniformly for $\alpha\in[0,1)$,  
  \[ \sum_{n\leq x} n^{\omega}\,\frac{\mathbbm{1}_{\N^k_{(c)}}(n)}{n^{1/ck}}\,e(n\alpha) = \frac{1}{c}\sum_{n\leq x} n^{\omega}\,\frac{\mathbbm{1}_{\N^k}(n)}{n^{1/k}}\,e(n\alpha) + O(x^{\max\{\omega-P(c,k),\,0\}+o(1)}). \]
 \end{lem}
 \begin{proof}
  Write $P=P(c,k)$, and define
  \[ g(\alpha;x) := \sum_{n\leq x} \mathbbm{1}_{\N^k_{(c)}}(n)\, e(n\alpha), \quad G(\alpha;x) := \frac{1}{c} \sum_{n\leq x} n^{(\frac{1}{c}-1)\frac{1}{k}} \mathbbm{1}_{\N^k}(n)\, e(n\alpha). \]
  By \eqref{PSW}, $R(\alpha;x):=g(\alpha;x)-G(\alpha;x)\ll x^{\frac{1}{ck}-P}$ uniformly for $\alpha\in[0,1)$. By partial summation applied to the sequence whose summatory is $R(\alpha;t)$, we have
  \begin{align*}
   &\sum_{n\leq x} n^{\omega}\,\frac{\mathbbm{1}_{\N^k_{(c)}}(n)}{n^{1/ck}}\,e(n\alpha)-\frac{1}{c}\sum_{n\leq x} n^{\omega}\,\frac{\mathbbm{1}_{\N^k}(n)}{n^{1/k}}\,e(n\alpha) \\
   &\hspace{+15em}= x^{\omega-\frac{1}{ck}} R(\alpha;x) - (\omega-\tfrac{1}{ck}) \int_1^x t^{\omega-\frac{1}{ck}-1} R(\alpha;t)\,\mathrm{d}t \\
   &\hspace{+15em}\ll x^{\omega-P}+\int_1^x t^{\omega-P-1}\,\mathrm{d}t.
  \end{align*}
  If $\omega>P$, the last expression is $\ll x^{\omega-P}$. If $\omega=P$, it is $\ll \log x = x^{o(1)}$. If $\omega<P$, it is $\ll 1$. This proves the lemma.
 \end{proof}
 
 The key point of Lemma \ref{trnsfr-PSS} is that \eqref{PSW} transfers exponential sums from $\N^k_{(c)}$ to $\N^k$, up to error. For $\omega<P(c,k)$ the $o(1)$ can be dropped, though this is immaterial here. The dyadic form (Lemma \ref{trnsfr-PSS2}), needed for Theorem \ref{main-PSS1}, gives extra cancellation for small $\omega$.

 \begin{lem}\label{trnsfr-PSS2}
  Suppose \eqref{PSW} holds, and let $\omega \in \R$. Then, for any $C \geq 1$ and uniformly for $\alpha \in [0,1)$,
  \[ \sum_{x/C \leq n\leq x} n^{\omega}\,\frac{\mathbbm{1}_{\N^k_{(c)}}(n)}{n^{1/ck}}\,e(n\alpha) = \frac{1}{c}\sum_{x/C\leq n\leq x} n^{\omega}\,\frac{\mathbbm{1}_{\N^k}(n)}{n^{1/k}}\,e(n\alpha) + O_C(x^{\omega-P(c,k)}). \]
 \end{lem}
 \begin{proof}
  The proof is similar to Lemma \ref{trnsfr-PSS}. Writing $P=P(c,k)$, $g$ and $G$ as before, $R(\alpha;x) := g(\alpha;x)-G(\alpha;x)\ll x^{\frac{1}{ck}-P}$ uniformly for $\alpha\in[0,1)$ by \eqref{PSW}. Put $y=\lceil x/C\rceil -1$. By partial summation,
  \begin{align*}
   &\sum_{y<n\leq x} n^{\omega}\,\frac{\mathbbm{1}_{\N^k_{(c)}}(n)}{n^{1/ck}}\,e(n\alpha)-\frac{1}{c}\sum_{y<n\leq x} n^{\omega}\,\frac{\mathbbm{1}_{\N^k}(n)}{n^{1/k}}\,e(n\alpha) \\
   &\hspace{+10em}= x^{\omega-\frac{1}{ck}} R(\alpha;x) - y^{\omega-\frac{1}{ck}} R(\alpha;y) - (\omega-\tfrac{1}{ck}) \int_y^x t^{\omega-\frac{1}{ck}-1}R(\alpha;t)\,\mathrm{d}t \\
   &\hspace{+10em}\ll_C x^{\omega-P}+\int_{\lceil x/C\rceil -1}^x t^{\omega-P-1}\,\mathrm{d}t \ll_C x^{\omega-P}. \qedhere
  \end{align*}
 \end{proof}
 
 Finally, we need an asymptotic for weighted solution counts over $\N^k$.
  
 \begin{thm}[{\cite[Theorem 4.2]{tafWWG}}]\label{wolem}
  Let $H_0=H_0(k)$ be as in \eqref{h0k1}, and let $h\geq 2H_0+1$. Let $\delta>0$ be any real number with
  \[ \delta < \frac{h-2H_0}{2h(h-1)H_0}. \]
  Then, for any $\omega \geq 1/h - \delta$,
  \[ \sum_{\substack{x_1,\ldots,x_h\in \N^k \\ x_1+\cdots+x_h = N}} (x_1\cdots x_h)^{\omega - \frac{1}{k}} = \mathfrak{S}_{k,h}(N)\frac{1}{k^h}\frac{\Gamma(\omega)^h}{\Gamma(h\omega)} N^{h\omega - 1} + O(N^{h\omega -1 -\tau}) \]
  for some $\tau = \tau(k,h,\delta) >0$, where the singular series $\mathfrak{S}_{k,h}(N)$ is defined as in \eqref{singser}.
 \end{thm}
 
 We are now ready to prove the main results of this section.
 
\subsection{Proof of Theorem \ref{main-PSS1}}
 Write $P=P(c,k)$, and assume
 \[ h\geq \max\{2H_0+1,\, P^{-1} + 1\}, \]
 Let $\omega \geq 1/h - \delta$, and define the exponential sums
 \[ T(\alpha;x) := \sum_{n\leq x} n^{\omega}\,\frac{\mathbbm{1}_{\N^k_{(c)}}(n)}{n^{1/ck}}\,e(n\alpha), \qquad T^{\sharp}(\alpha;x) := \sum_{x/h\leq n\leq x} n^{\omega}\,\frac{\mathbbm{1}_{\N^k_{(c)}}(n)}{n^{1/ck}}\,e(n\alpha). \]
 Since every solution of $x_1+\cdots+x_h = N$ has at least one variable $\geq N/h$, we have
 \[ \int_{0}^{1} (T(\alpha;N)-T^{\sharp}(\alpha;N))^{h}\, e(-N\alpha)\,\mathrm{d}\alpha = 0. \]
 Thus
 \begin{align}
  \sum_{\substack{x_1,\ldots,x_h\in \N^{k}_{(c)} \\ x_1+\cdots+x_h = N}} (x_1\cdots x_h)^{\omega-\frac{1}{ck}} &= \int_{0}^{1} \big(T(\alpha;N)^h - (T(\alpha;N)-T^{\sharp}(\alpha;N))^{h}\big)\, e(-N\alpha)\,\mathrm{d}\alpha \nonumber \\
  &= \sum_{j=1}^{h} (-1)^{j+1}\binom{h}{j} \int_{0}^{1} T^{\sharp}(\alpha;N)^j T(\alpha;N)^{h-j}\,e(-N\alpha)\,\mathrm{d}\alpha. \label{TsharpT}
 \end{align}
 
 By Lemmas \ref{trnsfr-PSS} and \ref{trnsfr-PSS2},
 \begin{equation*}
  T(\alpha;x) = \frac{1}{c}\, U(\alpha;x) + O(E(x)),\qquad T^{\sharp}(\alpha;x) = \frac{1}{c} \,U^{\sharp}(\alpha;x) + O(E^{\sharp}(x)) %\label{approxUT}
 \end{equation*}
 where
 \[ U(\alpha;x) := \sum_{n\leq x} n^{\omega}\,\frac{\mathbbm{1}_{\N^k}(n)}{n^{1/k}}\,e(n\alpha), \qquad U^{\sharp}(\alpha;x) := \sum_{x/h\leq n\leq x} n^{\omega}\,\frac{\mathbbm{1}_{\N^k}(n)}{n^{1/k}}\,e(n\alpha), \]
 \[ E(x) = x^{\max\{\omega-P,\,0\}+o(1)}, \qquad E^{\sharp}(x) = x^{\omega-P+o(1)}. \]
 Plugging these into \eqref{TsharpT}, each term becomes
 \begin{equation}
  \begin{aligned}
   \int_{0}^{1}\,&T^{\sharp}(\alpha;N)^j T(\alpha;N)^{h-j}\,e(-N\alpha)\,\mathrm{d}\alpha \\[-0.5em]
   &= \frac{1}{c^h} \int_{0}^{1} U^{\sharp}(\alpha;N)^j U(\alpha;N)^{h-j}\,e(-N\alpha)\,\mathrm{d}\alpha + O\Bigg( \underset{(r,s)\neq(0,0)}{\sum_{r=0}^{j}\sum_{s=0}^{h-j}} \lVert U^{\sharp}\rVert_{h-s}^{j-r} \lVert U\rVert_h^{h-j-s} (E^{\sharp})^r E^s \Bigg),
  \end{aligned} \label{TapproxU}
 \end{equation} 
 where H\"older's inequality upgrades the norms in the error term.
 
 Lemma \ref{whuaPPow} gives
 \[ \lVert U\rVert_h \ll N^{\max\{\omega - \frac{1}{h},\,0\}+o(1)},\qquad \lVert U^{\sharp}\rVert_{h-s} \ll N^{\omega - \frac{1}{(h-s)^{*}}+o(1)}, \]
 where $(h-s)^{*}:= \max\{h-s,2H_0\}$. Since $\omega\geq 1/h-\delta$, we have $\max\{\omega-\frac{1}{h},0\}\leq \omega-\frac{1}{h}+\delta$. Thus the terms in the error in \eqref{TapproxU} are
 \begin{align*}
  &\ll N^{ \left((j-r)\omega - \frac{j-r}{(h-s)^{*}}\right) \,+\, (h-j-s)\left(\omega-\frac{1}{h}+\delta\right) \,+\, r(\omega-P) \,+\, s\max\{\omega-P,\,0\} \,+\, o(1) } \\
  &= N^{h\omega - 1 - \left(\frac{j-r}{(h-s)^{*}} \,-\, \frac{j+s}{h} \,+\, rP \,+\, s\min\{P,\,\omega\} \,-\, (h-j-s)\delta\right) \,+\, o(1)}.
 \end{align*}
 Since $h\geq P^{-1}+1$ and $\omega\geq 1/h-\delta$, we have $\min\{P,\omega\}\geq 1/h - \delta$, so this is
 \[ \ll N^{h\omega - 1 - \left((j-r)(\frac{1}{(h-s)^{*}} - \frac{1}{h}) \,+\, r(P-\frac{1}{h}) \,-\, (h-j)\delta \right) \,+\, o(1) }. \]
 For $h \geq \max\{2H_0+1,\, P^{-1}+1\}$, the term $r(P-\frac{1}{h})$ is positive whenever $r\geq 1$, while if $r=0$ we still have positivity from $(j-r)(\frac{1}{(h-s)^{*}} - \frac{1}{h})$ since then $s\geq 1$. In either case the positive part is at least $\frac{1}{h(h-1)}$; since $h-j\leq h-1$ and $\delta<\frac{1}{h(h-1)^2}$ by definition, every term in the error in \eqref{TapproxU} is $O(N^{h\omega-1-\eta})$ for some $\eta = \eta(h,\delta)>0$.
 
 Substituting \eqref{TapproxU} into \eqref{TsharpT} gives
 \begin{align*}
  \sum_{\substack{x_1,\ldots,x_h\in \N^{k}_{(c)} \\ x_1+\cdots+x_h = N}} &(x_1\cdots x_h)^{\omega-\frac{1}{ck}} \\
  &= \frac{1}{c^h} \sum_{j=1}^{h} (-1)^{j+1}\binom{h}{j} \int_{0}^{1} U^{\sharp}(\alpha;N)^j U(\alpha;N)^{h-j}\,e(-N\alpha)\,\mathrm{d}\alpha + O(N^{h\omega-1-\eta}) \\
  &= \frac{1}{c^h} \sum_{\substack{x_1,\ldots,x_h\in \N^{k} \\ x_1+\cdots+x_h = N}} (x_1\cdots x_h)^{\omega-\frac{1}{k}} + O(N^{h\omega-1-\eta}).
 \end{align*}
 Since $\omega\geq 1/h - \delta$ and $0< \delta < \frac{1}{h(h-1)^2} \leq \frac{h-2H_0}{2h(h-1)H_0}$, Theorem \ref{wolem} applies, completing the proof (taking $\nu = \min\{\tau,\eta\}$).\hfill$\square$
 
%%%%%%%%%%%%%%%%%
\subsection{Proof of Theorem \ref{MTPSSkappa}}
 The set $\N^{k}_{(c)}$ has regularly varying counting function, since we have $|\N^{k}_{(c)}\cap[1,x]|=|\N_{(c)}\cap[1,x^{1/k}]|=x^{1/ck}+O(1)$. By Proposition \ref{PSShua}, it also satisfies \eqref{hua-type} with $H_0(\N^{k}_{(c)})=\max\{H_0(k),\,\lceil\tfrac{1}{2} P^{-1}\rceil\}$. Theorem \ref{main-PSS1} gives \eqref{mainest} with $\mathscr{S}=\N$ for $h\geq \max\{2H_0(k)+1,\, P^{-1} + 1\}$ (since $\mathfrak{S}_{k,h}(N) \asymp 1$ by \cite[Lemma 4.5]{tafWWG}). Hence, taking $h\geq \max\{2H_0(k)+1,\,2\lceil\tfrac{1}{2} P^{-1}\rceil+1\}$, Theorem \ref{main-tech} yields Theorem \ref{MTPSSkappa}. \hfill$\square$

%%%%%%%%%%%%%%%%%%%%%%%%%%%%%
\section{Piatetski-Shapiro prime powers}\label{sec-pspp}
 In this section we prove Theorem \ref{MTPSPP}. The arguments parallel those of Section \ref{sec-psw}, but here we work with $\P_{(c)}^k$. We begin with a Hua-type bound for $\P^k$. 
 
 \begin{lem}[{\cite[Lemma 5.1]{tafWWG}}]\label{PPhua}
  Let $k\geq 1$, and let $H_0 = H_0(k)$ be as in \eqref{h0k1}. Then, for every integer $h\geq H_0$,
  \[ \sum_{n\leq x} r_{\P^k,h}(n)^2 \ll x^{\frac{2h}{k} -1}(\log x)^{M} \]
  for some constant $M=M(k)>0$. When $k\geq 5$, the factor $(\log x)^M$ can be dropped.
 \end{lem}
  
 Next, we extend this estimate to weighted sums. While $x^{o(1)}$ bounds would suffice for our purposes, it is often useful in the $\P^k$ setting to make logarithmic losses explicit, so we state results in this form.
 
 \begin{lem}\label{whuaPP}
  Let $H_0 = H_0(k)$ and $M$ be as in Lemma \ref{PPhua}. For every $\ell \geq 1$, we have
  \[ \int_0^1 \bigg| \sum_{n \leq x} n^{\omega}\, \frac{\mathbbm{1}_{\P^k}(n)}{n^{1/k}} \,(\log n) \,e(n\alpha) \bigg|^{\ell} \, \mathrm{d}\alpha \ll x^{\max\{\ell\omega - \frac{\ell}{\ell^{*}},\,0\}} (\log x)^{\ell + 2H_0 + M}, \]
  where $\ell^{*} := \max\{\ell,2H_0\}$. Moreover, for any $C\geq 1$,
  \[ \int_0^1 \bigg| \sum_{x/C \leq n \leq x} n^{\omega}\, \frac{\mathbbm{1}_{\P^k}(n)}{n^{1/k}} \,(\log n) \,e(n\alpha) \bigg|^{\ell} \, \mathrm{d}\alpha \ll_C x^{\ell\omega - \frac{\ell}{\ell^{*}}} (\log x)^{2H_0 + M}. \]
 \end{lem}
 \begin{proof}
  Define $g(\alpha;x) := \sum_{n \leq x} \mathbbm{1}_{\P^k}(n)\,(\log n)\,e(n\alpha)$. We first claim that, uniformly for $2\leq t\leq x$,
  \[ \|g(\cdot;t)\|_{\ell}\ll t^{\frac{1}{k}-\frac{1}{\ell^*}}(\log t)^{\frac{2H_0+M}{\ell^*}}. \]
  Indeed, expanding the $2H_0$-th moment and using orthogonality, the integral counts solutions to
  \[ n_1+\cdots+n_{H_0}=m_1+\cdots+m_{H_0}, \qquad n_i,m_i\in \P^k\cap[1,t], \]
  with weight
  \[ (\log n_1)\cdots(\log n_{H_0})(\log m_1)\cdots(\log m_{H_0})\leq(\log t)^{2H_0}. \]
  Therefore, by Lemma \ref{PPhua},
  \[ \int_0^1 |g(\alpha;t)|^{2H_0}\,\mathrm{d}\alpha \leq (\log t)^{2H_0}\sum_{n\leq H_0t}r_{\P^k,H_0}(n)^2 \ll t^{\frac{2H_0}{k}-1}(\log t)^{2H_0+M}. \]
  Also, $\|g(\cdot;t)\|_{\infty}\leq \sum_{n\leq t}\mathbbm{1}_{\P^k}(n)\log n \ll t^{1/k}$. So if $\ell\leq 2H_0$, then $\ell^*=2H_0$, and the claim follows from the monotonicity of $L^p$-norms. If $\ell>2H_0$, then $\ell^{*} = \ell$, and
  \[ \|g(\cdot;t)\|_{\ell}^{\ell}\leq \|g(\cdot;t)\|_{\infty}^{\ell-2H_0}\|g(\cdot;t)\|_{2H_0}^{2H_0}\ll t^{\frac{\ell}{k}-1}(\log t)^{2H_0+M}, \]
  which proves the claim.

  By partial summation,
  \[ \sum_{n \leq x} n^{\omega}\,\frac{\mathbbm{1}_{\P^k}(n)}{n^{1/k}}\,(\log n)\,e(n\alpha) = x^{\omega-\frac{1}{k}} g(\alpha;x)-(\omega-\tfrac{1}{k})\int_1^x t^{\omega-\frac{1}{k}-1}g(\alpha;t)\,\mathrm{d}t. \]
  Hence, by the triangle inequality,
  \begin{align*}
   \bigg\|\sum_{n \leq x} n^{\omega}\,\frac{\mathbbm{1}_{\P^k}(n)}{n^{1/k}}\,(\log n)\,e(n\alpha)\bigg\|_{\ell}
   &\ll x^{\omega-\frac{1}{k}}\|g(\cdot;x)\|_{\ell}+\int_1^x t^{\omega-\frac{1}{k}-1}\|g(\cdot;t)\|_{\ell}\,\mathrm{d}t \\
   &\ll x^{\omega-\frac{1}{\ell^*}}(\log x)^{\frac{2H_0+M}{\ell^*}}+\int_1^x t^{\omega-\frac{1}{\ell^*}-1}(\log t)^{\frac{2H_0+M}{\ell^*}}\,\mathrm{d}t.
  \end{align*}
  If $\omega>1/\ell^*$, the last expression is $\ll x^{\omega-\frac{1}{\ell^*}}(\log x)^{\frac{2H_0+M}{\ell^*}}$. If $\omega=1/\ell^*$, it is $\ll (\log x)^{1+\frac{2H_0+M}{\ell^*}}$. If $\omega<1/\ell^*$, it is $\ll 1$. Thus, in all cases,
  \[ \bigg\|\sum_{n \leq x} n^{\omega}\,\frac{\mathbbm{1}_{\P^k}(n)}{n^{1/k}}\,(\log n)\,e(n\alpha)\bigg\|_{\ell}\ll x^{\max\{\omega-\frac{1}{\ell^*},\,0\}}(\log x)^{1+\frac{2H_0+M}{\ell^*}}. \]
  Raising to the power $\ell$, and using $\ell/\ell^*\leq 1$, gives the result.

  For the short interval estimate, put $y:=\lceil x/C\rceil-1$. By partial summation,
  \begin{align*}
   &\sum_{x/C\leq n\leq x} n^{\omega}\,\frac{\mathbbm{1}_{\P^k}(n)}{n^{1/k}}\,(\log n)\,e(n\alpha) \\
   &\hspace{10em}= x^{\omega-\frac{1}{k}}g(\alpha;x) - y^{\omega-\frac{1}{k}}g(\alpha;y)-(\omega-\tfrac{1}{k})\int_y^x t^{\omega-\frac{1}{k}-1}g(\alpha;t)\,\mathrm{d}t.
  \end{align*}
  Since $y\asymp_C x$, the same argument gives
  \begin{align*}
   &\bigg\|\sum_{x/C\leq n\leq x} n^{\omega}\,\frac{\mathbbm{1}_{\P^k}(n)}{n^{1/k}}\,(\log n)\,e(n\alpha)\bigg\|_{\ell} \\
   &\hspace{+12em} \ll_C x^{\omega-\frac{1}{\ell^*}}(\log x)^{\frac{2H_0+M}{\ell^*}}+\int_{\lceil x/C\rceil-1}^{x} t^{\omega-\frac{1}{\ell^*}-1}(\log x)^{\frac{2H_0+M}{\ell^*}}\,\mathrm{d}t \\
   &\hspace{+12em} \ll_C x^{\omega-\frac{1}{\ell^*}}(\log x)^{\frac{2H_0+M}{\ell^*}}.
  \end{align*}
  Raising to the power $\ell$, and using $\ell/\ell^*\leq 1$, finishes the proof.
 \end{proof}
 
 When $\omega < 1/\ell^{*}$, the logarithmic factor in the first bound can be omitted. This minor sharpening will not be needed, since logarithmic losses are harmless for our purposes.
 
 We now obtain a Hua-type estimate for $\P^k_{(c)}$.
 
 \begin{prop}[Hua-type bound for $\P^k_{(c)}$]\label{PSPhua}
  Suppose \eqref{PSWG} holds, and let $H_0 = H_0(k)$ and $M$ be as in Lemma \ref{PPhua}. Then, for every $h \geq \max\{H_0, \frac{1}{2}P^{*}(c,k)^{-1} \}$, we have
  \[ \sum_{n\leq x} r_{\P^k_{(c)},h}(n)^2 \ll x^{\frac{2h}{ck}-1}(\log x)^{2h + 2H_0 + M}. \]
 \end{prop} 
 \begin{proof}
  Write $g(\alpha;x) := \sum_{n\leq x}\mathbbm{1}_{\P^k_{(c)}}(n)\,(\log n)\, e(n\alpha)$. Since every element of $\P^k_{(c)}$ is at least $2^k$, the logarithmic weights are bounded below by a positive constant depending only on $k$. Hence, by orthogonality,
  \[ \sum_{n\leq x} r_{\P^k_{(c)},h}(n)^2 \ll \int_0^1 |g(\alpha;x)|^{2h}\,\mathrm{d}\alpha. \]
  By \eqref{PSWG} and the inequality $(x+y)^{2h}\ll_h x^{2h} + y^{2h}$,
  \begin{align*}
   \int_0^1 |g(\alpha;x)|^{2h}\,\mathrm{d}\alpha
   &\ll \int_0^1 \bigg|\frac{1}{c}\sum_{n\leq x} n^{(\frac{1}{c}-1)\frac{1}{k}}\,\mathbbm{1}_{\P^k}(n)\,(\log n)\,e(n\alpha)\bigg|^{2h}\,\mathrm{d}\alpha + x^{\frac{2h}{ck}-2hP^*}.
  \end{align*}
  Applying Lemma \ref{whuaPP} with $\ell=2h$ and $\omega=1/(ck)$ gives
  \[ \int_0^1 \bigg|\frac{1}{c}\sum_{n\leq x} n^{(\frac{1}{c}-1)\frac{1}{k}}\,\mathbbm{1}_{\P^k}(n)\,(\log n)\,e(n\alpha)\bigg|^{2h}\,\mathrm{d}\alpha \ll x^{\frac{2h}{ck}-1}(\log x)^{2h+2H_0+M}. \]
  Also, since $h\geq \frac{1}{2}P^*(c,k)^{-1}$, we have $x^{\frac{2h}{ck}-2hP^*}\leq x^{\frac{2h}{ck}-1}$.
  Combining these estimates completes the proof.
 \end{proof}
 
 To apply Theorem \ref{main-tech}, we make use of the following analogue of \cite[Theorem 1.2]{tafWWG} for $\P_{(c)}^k$.
 
 \begin{thm}\label{main-PSP}
  Suppose \eqref{PSWG} holds, let $H_0 = H_0(k)$ be as in \eqref{h0k1}, and let $h\geq \max\{2H_0+1, P^{*}(c,k)^{-1} + 1\}$. Let $\delta >0$ be any real number with 
  \[ \delta < \frac{1}{h(h-1)^2}. \]
  Then, for any $\omega \geq 1/h - \delta$,    
  \begin{equation*}
   \begin{aligned}
    \sum_{\substack{x_1,\ldots,x_h \in \P^k_{(c)} \\ x_1+\cdots+x_h = N}} (x_1\cdots x_h)^{\omega-\frac{1}{ck}}(\log x_1 \cdots \log x_h) 
    = \mathfrak{S}^{*}_{k,h}(N)\frac{1}{c^h}\frac{\Gamma(\omega)^h}{\Gamma(h\omega)} N^{h\omega-1} + O_R\bigg(\frac{N^{h\omega -1}}{(\log N)^{R}}\bigg)
   \end{aligned}
  \end{equation*}
  for every $R>1$, where the singular series $\mathfrak{S}^{*}_{k,h}(N)$ is defined as in \eqref{singserp}.
 \end{thm}

 As in the case of $\N^k$, the proof of Theorem \ref{main-PSP} will be based on \eqref{PSWG}, which we first generalize to weighted sums.
 
 \begin{lem}\label{PSWGequiv}
  Suppose \eqref{PSWG} holds, and let $\omega \in \R$. Then, uniformly for $\alpha\in[0,1)$,  
  \begin{equation*}
   \sum_{n\leq x} n^{\omega}\,\frac{\mathbbm{1}_{\P^k_{(c)}}(n)}{n^{1/ck}}\, (\log n)\, e(n\alpha) 
   = \frac{1}{c} \sum_{n\leq x} n^{\omega}\,\frac{\mathbbm{1}_{\P^k}(n)}{n^{1/k}}\, (\log n)\, e(n\alpha) + O(x^{\max\{\omega-P^{*}(c,k),\,0\}}\log x).
  \end{equation*}
  Moreover, for any $C\geq 1$,
  \begin{equation*}
   \sum_{x/C\leq n\leq x} n^{\omega}\,\frac{\mathbbm{1}_{\P^k_{(c)}}(n)}{n^{1/ck}}\,(\log n)\, e(n\alpha) 
   = \frac{1}{c} \sum_{x/C\leq n\leq x} n^{\omega}\,\frac{\mathbbm{1}_{\P^k}(n)}{n^{1/k}}\,(\log n)\, e(n\alpha) 
   + O_C(x^{\omega-P^{*}(c,k)}).
  \end{equation*}
 \end{lem}
 \begin{proof}
  The proof is similar to the proofs of Lemmas \ref{trnsfr-PSS} and \ref{trnsfr-PSS2}. Write $P^{*}=P^{*}(c,k)$, and set
  \[ g(\alpha;x) = \sum_{n\leq x} \mathbbm{1}_{\P^k_{(c)}}(n)\, (\log n)\,e(n\alpha), \quad G(\alpha;x) = \frac{1}{c}\sum_{n\leq x} n^{(\frac{1}{c}-1)\frac{1}{k}}\,\mathbbm{1}_{\P^k}(n)\, (\log n)\,e(n\alpha). \]
  By \eqref{PSWG}, $R(\alpha;x):=g(\alpha;x)-G(\alpha;x)\ll x^{\frac{1}{ck}-P^{*}}$ uniformly for $\alpha\in[0,1)$. Put $W(t):=t^{\omega-\frac{1}{ck}}$. By partial summation,
  \begin{align*}
   &\sum_{n\leq x} n^{\omega}\,\frac{\mathbbm{1}_{\P^k_{(c)}}(n)}{n^{1/ck}}\,(\log n)\,e(n\alpha)-\frac{1}{c}\sum_{n\leq x} n^{\omega}\,\frac{\mathbbm{1}_{\P^k}(n)}{n^{1/k}}\,(\log n)\,e(n\alpha) \\
   &\hspace{+15em}= x^{\omega-\frac{1}{ck}}R(\alpha;x)-(\omega-\tfrac{1}{ck})\int_1^x t^{\omega-\frac{1}{ck}-1}R(\alpha;t)\,\mathrm{d}t \\
   &\hspace{+15em}\ll x^{\omega-P^{*}}+\int_1^x t^{\omega-P^{*}-1}\,\mathrm{d}t.
  \end{align*}
  If $\omega>P^{*}$, the last expression is $\ll x^{\omega-P^{*}}$. If $\omega=P^{*}$, it is $\ll \log x$. If $\omega<P^{*}$, it is $\ll 1$. Hence it is always $O(x^{\max\{\omega-P^{*},\,0\}}\log x)$.

  For the interval estimate, put $y=\lceil x/C\rceil - 1$. Partial summation gives
  \begin{align*}
   &\sum_{y<n\leq x} n^{\omega}\,\frac{\mathbbm{1}_{\P^k_{(c)}}(n)}{n^{1/ck}}\,(\log n)\,e(n\alpha)-\frac{1}{c}\sum_{y<n\leq x} n^{\omega}\,\frac{\mathbbm{1}_{\P^k}(n)}{n^{1/k}}\,(\log n)\,e(n\alpha) \\
   &\hspace{+10em}= x^{\omega-\frac{1}{ck}}R(\alpha;x) - y^{\omega-\frac{1}{ck}} R(\alpha;y) - (\omega-\tfrac{1}{ck})\int_y^x t^{\omega-\frac{1}{ck}-1} R(\alpha;t)\,\mathrm{d}t \\
   &\hspace{+10em}\ll_C x^{\omega-P^{*}}+\int_{\lceil x/C\rceil-1}^x t^{\omega-P^{*}-1}\,\mathrm{d}t \\
   &\hspace{+10em}\ll_C x^{\omega-P^{*}}. \qedhere
  \end{align*}
 \end{proof}

 Note that the extra factor of $\log$ in the error term of Lemma \ref{PSWGequiv} arises only in the borderline case $\omega = P^{*}(c,k)$; for $\omega \neq P^{*}(c,k)$ it may be omitted. This will not affect what follows.
 
 The final ingredient is an asymptotic formula for weighted representation counts of $\P^k$. 
  
 \begin{thm}[{\cite[Theorem 1.2]{tafWWG}}]\label{taflem2}
  Let $k\geq 1$, let $H_0=H_0(k)$ be as in \eqref{h0k1}, and let $h\geq 2H_0(k)+1$. Let $\delta>0$ be any real number with 
  \[ \delta < \frac{h-2H_0}{2 h(h-1) H_0}. \]
  Then, for any $\omega \geq 1/h - \delta$,
  \begin{equation*}
   \begin{aligned}
    \sum_{\substack{x_1,\ldots,x_h \in \P^k \\ x_1+\cdots+x_h = N}} (x_1\cdots x_h)^{\omega-\frac{1}{k}}(\log x_1 \cdots \log x_h) = \mathfrak{S}^{*}_{k,h}(N)\,\frac{\Gamma(\omega)^h}{\Gamma(h\omega)} \, N^{h\omega-1} + O_R\bigg(\frac{N^{h\omega -1}}{(\log N)^{R}}\bigg)
   \end{aligned}
  \end{equation*}
  for every $R>1$, where the singular series $\mathfrak{S}^{*}_{k,h}(N)$ is defined as in \eqref{singserp}.
 \end{thm}

 We are now ready to prove Theorem \ref{main-PSP}. The argument follows the same general outline as the proof of Theorem \ref{main-PSS1}, hence the proof is streamlined.
 
%%%%% 
\subsection{Proof of Theorem \ref{main-PSP}}
 Put $P^{*}=P^{*}(c,k)$, and fix
 \[ h\geq \max\{2H_0(k)+1,\,(P^{*})^{-1}+1\}. \]
 Let $\omega\geq 1/h - \delta$, and define the exponential sums
 \[ T(\alpha;x) := \sum_{n\leq x} n^{\omega}\,\frac{\mathbbm{1}_{\P^k_{(c)}}(n)}{n^{1/ck}}\,(\log n)\,e(n\alpha),\quad T^{\sharp}(\alpha;x) := \sum_{x/h\leq n\leq x} n^{\omega}\,\frac{\mathbbm{1}_{\P^k_{(c)}}(n)}{n^{1/ck}}\,(\log n)\,e(n\alpha), \]
 \[ U(\alpha;x) := \sum_{n\leq x} n^{\omega}\,\frac{\mathbbm{1}_{\P^k}(n)}{n^{1/k}}\,(\log n)\,e(n\alpha),\quad U^{\sharp}(\alpha;x) := \sum_{x/h\leq n\leq x} n^{\omega}\,\frac{\mathbbm{1}_{\P^k}(n)}{n^{1/k}}\,(\log n)\,e(n\alpha). \]
 Since any representation $x_1+\cdots+x_h=N$ contains a summand $\geq N/h$, orthogonality yields
 \begin{equation}
  \begin{aligned}
   \sum_{\substack{x_1,\ldots,x_h \in \P_{(c)}^k \\ x_1+\cdots+x_h = N}} (x_1\cdots x_h)^{\omega-\frac{1}{ck}}&(\log x_1 \cdots \log x_h) \\
   &= \sum_{j=1}^{h} (-1)^{j+1}\binom{h}{j} \int_{0}^{1} T^{\sharp}(\alpha;N)^j T(\alpha;N)^{h-j}\,e(-N\alpha)\,\mathrm{d}\alpha. 
  \end{aligned}\label{Texpand}
 \end{equation}
 Lemma \ref{PSWGequiv} supplies the approximations
 \[ T(\alpha;x) = \frac{1}{c}\,U(\alpha;x)+O(E(x)),\qquad T^{\sharp}(\alpha;x) = \frac{1}{c}\,U^\sharp(\alpha;x)+O(E^\sharp(x)), \]
 with
 \[ E(x)\ll x^{\max\{\omega-P^{*},0\}}\log x,\qquad E^\sharp(x)\ll x^{\omega-P^{*}}, \]
 uniformly in \(\alpha\). Inserting these into \eqref{Texpand} and expanding gives
 \begin{align*}
  \int_{0}^{1}&\,T^{\sharp}(\alpha;N)^j T(\alpha;N)^{h-j}\,e(-N\alpha)\,\mathrm{d}\alpha \\[-0.5em]
   &= \frac{1}{c^h} \int_{0}^{1} U^{\sharp}(\alpha;N)^j U(\alpha;N)^{h-j}\,e(-N\alpha)\,\mathrm{d}\alpha + O\Bigg( \underset{(r,s)\neq(0,0)}{\sum_{r=0}^{j}\sum_{s=0}^{h-j}} \lVert U^{\sharp}\rVert_{h-s}^{j-r} \lVert U\rVert_h^{h-j-s} (E^{\sharp})^r E^s \Bigg).
 \end{align*} 

 From Lemma \ref{whuaPP} we know
 \[ \lVert U\rVert_h \ll N^{\max\{\omega - \frac{1}{h},0\}+o(1)},\qquad \lVert U^{\sharp}\rVert_{h-s} \ll N^{\omega - \frac{1}{(h-s)^{*}}+o(1)}, \]
 where $(h-s)^{*}:=\max\{h-s,2H_0\}$ and the $N^{o(1)}$ hides powers of $\log$. As in the proof of Theorem \ref{main-PSS1}, a short calculation shows that each term in the error is
 \[ \ll N^{h\omega - 1 - \left((j-r)(\frac{1}{(h-s)^{*}} - \frac{1}{h}) \,+\, r(P^{*}-\frac{1}{h}) \,-\, (h-j)\delta \right) \,+\, o(1) }. \]
 For $h \geq \max\{2H_0+1,\, (P^{*})^{-1}+1\}$, the term $r(P^{*}-\frac{1}{h}) \geq \frac{1}{h(h-1)}$ whenever $r\geq 1$, while if $r=0$ we still have $(j-r)(\frac{1}{(h-s)^{*}} - \frac{1}{h}) \geq \frac{1}{h(h-1)}$ since then $s\geq 1$. Since $h-j\leq h-1$ and $\delta<\frac{1}{h(h-1)^2}$ by definition, every term in the error is $O(N^{h\omega-1-\eta})$ for some $\eta = \eta(h,\delta)>0$.
 
 Substituting back into \eqref{Texpand} we conclude
 \begin{align*}
  \sum_{\substack{x_1,\ldots,x_h \in \P_{(c)}^k \\ x_1+\cdots+x_h = N}} (x_1&\cdots x_h)^{\omega-\frac{1}{ck}}(\log x_1 \cdots \log x_h) \\
  &= \frac{1}{c^h} \sum_{j=1}^{h} (-1)^{j+1}\binom{h}{j} \int_{0}^{1} U^{\sharp}(\alpha;N)^j U(\alpha;N)^{h-j}\,e(-N\alpha)\,\mathrm{d}\alpha + O(N^{h\omega-1-\eta}) \\
  &= \frac{1}{c^h} \sum_{\substack{x_1,\ldots,x_h \in \P^k \\ x_1+\cdots+x_h = N}} (x_1\cdots x_h)^{\omega-\frac{1}{k}}(\log x_1 \cdots \log x_h) + O(N^{h\omega-1-\eta}),
 \end{align*}
 so invoking Theorem \ref{taflem2} (since $\delta < \frac{1}{h(h-1)^2} \leq \frac{h-2H_0}{2h(h-1)H_0}$) completes the proof.\hfill$\square$
 
 \begin{rem}[Case $\P_{(c)}$]\label{rem-k1}
  When $k=1$, sharper bounds on the least admissible $h$ can be obtained for small $c$, using the identity\footnote{From which one deduces
  $\int_{0}^{1} T^h\,\mathrm{d}\alpha = \int_{0}^{1} U^h\,\mathrm{d}\alpha + O\big((\sup_{\alpha\in[0,1)} |T-U|)\,(\sum_{j=0}^{h-1} \|T\|_{h-1}^j \|U\|_{h-1}^{h-1-j})\big)$.}
  \[ T^h = U^h + (T-U)\bigg(\sum_{j=0}^{h-1} T^j U^{h-1-j}\bigg). \]
  In particular, for $1<c<53/50$, Kumchev \cite[Theorem 2]{kum97} applied with $\delta = 1-1/c$ gives
  \[ \sum_{n\leq x}\mathbbm{1}_{\P_{(c)}}(n)\, e(n\alpha) = \frac{1}{c} \sum_{n\leq x} n^{1/c-1}\,\mathbbm{1}_{\P}(n)\,e(n\alpha) + O_{\eps}(x^{2/c-1-\eps}) \]
  uniformly in $\alpha$. It follows inductively that
  \[ \int_{0}^{1} \bigg|\sum_{n\leq x} \mathbbm{1}_{\P_{(c)}}(n)\, e(n\alpha) \bigg|^h\,\mathrm{d}\alpha \ll x^{h/c - 1 + o(1)} \]
  for every $h\geq 3$. Hence, for such $c$, $\P_{(c)}$ satisfies \eqref{hua-type} with $H_0(\P_{(c)})=2$. This in turn allows one to deduce Theorem \ref{main-PSP} for all $h\geq 3$ when $k=1$, but only under the weaker condition $\omega \geq 1/(h-1)-\eps$ (thus not reaching the critical case $\omega = 1/h$, required for thin subbases).
 \end{rem}

\subsection{Proof of Theorem \ref{MTPSPP}}
 The Piatetski-Shapiro prime number theorem
 \[ |\P^{k}_{(c)}\cap[1,x]| = |\P_{(c)}\cap[1,x^{1/k}]| \sim \frac{kx^{1/ck}}{\log x} \]
 follows from \eqref{PSWG} by partial summation (setting $\alpha=0$), so $\P^{k}_{(c)}$ has regularly varying counting function. By Proposition \ref{PSPhua}, it also satisfies \eqref{hua-type} with $H_0(\P^{k}_{(c)})=\max\{H_0(k),\,\lceil\tfrac{1}{2} (P^{*})^{-1}\rceil\}$. Theorem \ref{main-PSP} gives \eqref{mainest} with $\mathscr{S}=K(k)\cdot\N+h$ for $h\geq \max\{2H_0(k)+1, (P^{*})^{-1}+1\}$, where $K(k)$ is defined as in \eqref{defkk} (since $\mathfrak{S}^{*}_{k,h}(N) \asymp 1$ for $N\equiv h\pmod{K(k)}$ by \cite[Lemma 5.6]{tafWWG}). Hence, taking $h\geq \max\{2H_0(k)+1,\,2\lceil\tfrac{1}{2} (P^{*})^{-1}\rceil+1\}$, Theorem \ref{main-tech} yields Theorem \ref{MTPSPP}. \hfill$\square$

%%%%%%%%%%%%%%%%%%%%%%
\addtocontents{toc}{\protect\setcounter{tocdepth}{0}}
\section*{Acknowledgements}
 I would like to thank Sam Chow for suggesting, during my PhD defense, that I consider the problem of thin subbases in Piatetski-Shapiro sequences. I acknowledge the support of the S\~ao Paulo Research Foundation (FAPESP), Brazil, Process No.~2025/15961-3.
 
\addtocontents{toc}{\protect\setcounter{tocdepth}{1}}
%%%%%%%%%%%%%%%%%%%%%%

% ----------------------------------------------------------------
\bibliographystyle{amsplain}
\bibliography{$HOME/Academie/Recherche/_latex/bibliotheca}%
\end{document}